\documentclass[oneside,english]{amsart}
\usepackage[T1]{fontenc}
\usepackage[latin9]{inputenc}
\usepackage{geometry}
\usepackage{color}
\usepackage{babel}
\usepackage{verbatim}
\usepackage{amstext}
\usepackage{amsthm}
\usepackage{amssymb}
\usepackage{setspace}
\usepackage{esint}
\usepackage[unicode=true,pdfusetitle,
 bookmarks=true,bookmarksnumbered=true,bookmarksopen=true,bookmarksopenlevel=1,
 breaklinks=false,pdfborder={0 0 1},backref=section,colorlinks=true]
 {hyperref}

\makeatletter
\numberwithin{equation}{section}
\numberwithin{figure}{section}
 \theoremstyle{definition}
 \newtheorem*{defn*}{\protect\definitionname}
\theoremstyle{plain}
\newtheorem{thm}{\protect\theoremname}[section]
  \theoremstyle{plain}
  \newtheorem{prop}[thm]{\protect\propositionname}
  \theoremstyle{plain}
  \newtheorem{conjecture}[thm]{\protect\conjecturename}
  \theoremstyle{plain}
  \newtheorem{cor}[thm]{\protect\corollaryname}
  \theoremstyle{plain}
  \newtheorem{lem}[thm]{\protect\lemmaname}
  \theoremstyle{remark}
  \newtheorem{rem}[thm]{\protect\remarkname}
  \theoremstyle{definition}
  \newtheorem{defn}[thm]{\protect\definitionname}

\makeatother

  \providecommand{\conjecturename}{Conjecture}
  \providecommand{\corollaryname}{Corollary}
  \providecommand{\definitionname}{Definition}
  \providecommand{\lemmaname}{Lemma}
  \providecommand{\propositionname}{Proposition}
  \providecommand{\remarkname}{Remark}
\providecommand{\theoremname}{Theorem}

\begin{document}

\title{Cutoff on Hyperbolic Surfaces}

\author{Konstantin Golubev and Amitay Kamber}

\thanks{Konstantin Golubev, Bar-Ilan University and Weizmann Institute of
Science, k.golubev@gmail.com\\
Amitay Kamber, Einstein Institute of Mathematics, The Hebrew University
of Jerusalem, amitay.kamber@mail.huji.ac.il}
\begin{abstract}
In this paper we study the common distance between points and the
behavior of a constant length step discrete random walk on finite
area hyperbolic surfaces. We show that if the second smallest eigenvalue
of the Laplacian is at least $1/4$, then the distances on the surface
are highly concentrated around the minimal possible value, and that
the discrete random walk exhibits cutoff. This extends the results
of Lubetzky and Peres (\cite{lubetzky2015cutoff}) from the setting
of Ramanujan graphs to the setting of hyperbolic surfaces. By utilizing
density theorems of exceptional eigenvalues from \cite{sarnak1991bounds},
we are able to show that the results apply to congruence subgroups
of $SL_{2}\left(\mathbb{Z}\right)$ and other arithmetic lattices,
without relying on the well known conjecture of Selberg (\cite{selberg1965estimation}).

Conceptually, we show the close relation between the cutoff phenomenon
and temperedness of representations of algebraic groups over local
fields, partly answering a question of Diaconis (\cite{diaconis1996cutoff}),
who asked under what general phenomena cutoff exists.
\end{abstract}

\maketitle
\global\long\def\C{\mathbb{C}}
\global\long\def\R{\mathbb{R}}
\global\long\def\Q{\mathbb{Q}}
\global\long\def\N{\mathbb{N}}
\global\long\def\Z{\mathbb{Z}}
\global\long\def\O{\mathcal{O}}
\global\long\def\H{\mathbb{H}}

\global\long\def\n#1{\left\Vert #1\right\Vert _{1}}
\global\long\def\nn#1{\left\Vert #1\right\Vert _{2}}

\section{Introduction}

Let $\H$ be the hyperbolic plane equipped with the standard metric
$d$ and the standard measure $\mu$. Let $\Gamma\subset PSL_{2}(\R)$
be a lattice and let $X=\Gamma\backslash\H$ the the quotient space,
which is a hyperbolic surface if $\text{\ensuremath{\Gamma}}$ is
torsion-free, and an orbifold in general. The measure $\mu$ descends
to a finite measure on $X$, and let $d_{X}:X\times X\to\R_{\ge0}$
be the induced distance on $X$. The injectivity radius of a point
$x_{0}\in X$ is $\frac{1}{2}\inf_{1\ne\gamma\in\Gamma}d\left(\tilde{x}_{0},\gamma\tilde{x}_{0}\right)$,
where $\tilde{x}_{0}\in\H$ is a lift of $x_{0}$ to $\H$. Denote
$R_{X}=\text{acosh}\left(\mu\left(X\right)/2\pi+1\right)$. This is
the radius of the hyperbolic ball whose volume equals the volume $\mu\left(X\right)$
of $X$. 
\begin{defn*}
We say that $X=\Gamma\backslash\H$ is Ramanujan\footnote{It seems that the notion of a Ramanujan surface (or more generally,
a Ramanuajan manifold or a Ramanujan orbifold) does not appear in
literature, but is natural given the standard notions of a Ramanujan
graph (\cite{lubotzky1988ramanujan}) and a Ramanujan complex (\cite{lubotzky2005ramanujan}).} if the non-trivial spectrum of the Laplacian on $L^{2}\left(X\right)$
is bounded from below by $1/4$. Equivalently, every non-trivial subrepresentation
of $G$ on $L^{2}\left(\Gamma\backslash G\right)$ with $K=PSO_{2}\left(\mathbb{R}\right)$-fixed
vectors is tempered. 
\end{defn*}
We write $C=C(t)$ if $C$ is a constant depending only on $t$. We
write $a\ll_{t}b$ if there is $C=C(t)$ such that $a\le C\cdot b$
holds, and $a\asymp_{t}b$ if both $a\ll_{t}b$ and $b\ll_{t}a$ take
place.

\subsection*{Common Distance}
\begin{thm}
\label{thm1}Let $\Gamma\subset PSL_{2}(\R)$ be a lattice, $X=\Gamma\backslash\H$
, and assume $R_{X}=\text{acosh}\left(\mu\left(X\right)/2\pi+1\right)\ge1$.
Then for a point $x_{0}\in X$ and for all $\gamma>0$, the following
inequality holds
\[
\mu\left(x\in X:d_{X}\left(x_{0},x\right)\le R_{X}-\gamma\ln\left(R_{X}\right)\right)/\mu\left(X\right)\ll R_{X}^{-\gamma}.
\]
If $X$ is Ramanujan and $x_{0}\in X$ has injectivity radius at least
$r_{0}$, then for all $\gamma>0$, the following inequality holds
\[
\mu\left(x\in X:d_{X}\left(x_{0},x\right)\ge R_{X}+\gamma\ln\left(R_{X}\right)\right)/\mu\left(X\right)\ll_{r_{0}}\left(1+\gamma^{2}\right)R_{X}^{2-\gamma}.
\]
\end{thm}

In other words, for a point $x_{0}$ on a Ramanujan surface $X$,
the distance from it to almost every other point is approximately
$R_{X}$, with the window of size $\left(2+\epsilon\right)\ln\left(R_{X}\right)$.
We emphasize that the constants in the theorem do not depend on the
surface, and hence the result is interesting for a sequence of Ramanujan
quotients with volume increasing to infinity, which is not known to
exist. However, the well known conjecture of Selberg asserts that
the quotients defined by the congruence subgroups of $SL_{2}\left(\mathbb{Z}\right)$
form a sequence of such quotients (see \cite{selberg1965estimation,sarnak1995selberg}
and also Theorem \ref{thm:SL2Z} below). Alternatively, one may conjecture
that as in case of graphs, a ``random'' surface is almost Ramanujan
with a proper choice of the random model (see Conjecture \ref{conj:comb conjectures}
below).

\subsection*{Cutoff of Random Walks}

In the second result, we consider the speed of convergence in the
$L^{1}$-norm of two different random walks on $X$. The first one
is the hyperbolic Brownian motion on $X$, which we consider as an
operator $B_{t}:C(X)\to C(X)$ for $t\in\mathbb{R}_{\ge0}$, where
$C\left(X\right)$ is the space of continuous functions on $X$. The
second one is the discrete time random walk with step of a fixed length,
i.e., at each step the walker rotates at a uniformly chosen angle
and makes a step of some fixed length $r_{1}>0$. The corresponding
operator $A_{r_{1}}:C(X)\to C(X)$ is the distance $r_{1}$ averaging
operator. By duality, we consider both random walks as acting on measures
on $X$.

Specifically, for a point $x_{0}\in X$ consider the continuous time
random walk $B_{t}\delta_{x_{0}}$, and the discrete time random walk
$A_{r_{1}}^{k}\delta_{x_{0}}$, both considered as measures on $X$.
One can show that the measures defined by the two random walks, for
$t\gg0$ or $k\gg0$, are defined by some $L^{1}$-functions, which
converge in the $L^{1}$-norm to the constant function $\pi$ on $X$
normalized as $\pi\left(x\right)=\mu\left(X\right)^{-1}$ for all
$x\in X$. The following theorem gives an exact estimate on the rate
of convergence for points with injectivity radius bounded away from
$0$. 
\begin{thm}
\label{thm2}Fix $0<r_{0}$, $0<r_{1}$, assume that $X=\Gamma\backslash\H$
is Ramanujan, and let $x_{0}\in X$ be a point with injectivity radius
at least $r_{0}$.
\begin{enumerate}
\item There exist constants $c=c(r_{1})>0$, and $C=C\left(r_{0},r_{1}\right)$,
such that
\begin{enumerate}
\item If $k$ satisfies $k\alpha r_{1}<R_{X}-\lambda\sqrt{R_{X}}$ then
$\n{A_{r_{1}}^{k}\delta_{x_{0}}-\pi}>2-Ce^{-c\lambda^{2}}$;
\item If $k$ satisfies $k\alpha r_{1}>R_{X}+\lambda\sqrt{R_{X}}$ then
$\n{A_{r}^{k}\delta_{x_{0}}-\pi}<Ce^{-c\lambda^{2}}$;
\end{enumerate}
for every $\lambda>0$, assuming $R_{X}\gg_{r_{0},r_{1},\lambda}1$
, and where $\alpha=\frac{1}{\pi r_{1}}\intop_{0}^{\pi}\ln\left(e^{r_{1}}\cos^{2}\theta+e^{-r_{1}}\sin^{2}\theta\right)d\theta\in(0,1)$.
\item There exist constants $c>0$, $C=C(r_{0})$ such that
\begin{enumerate}
\item If $t$ satisfies $t<R_{X}-\lambda\sqrt{R_{X}}$ then $\n{B_{t}\delta_{x_{0}}-\pi}>2-Ce^{-c\lambda^{2}}$;
\item If $t$ satisfies $t>R_{X}+\lambda\sqrt{R_{X}}$ then$\n{B_{t}\delta_{x_{0}}-\pi}<Ce^{-c\lambda^{2}}$;
\end{enumerate}
for every $\lambda>0$, assuming $R_{X}\gg_{r_{0},\lambda}1$ .
\end{enumerate}
\end{thm}

As in Theorem \ref{thm1}, the lower bounds $(1a)$ and $(2a)$ do
not exploit the assumption that $X$ is Ramanujan nor the assumption
that the injectivity radius of $x_{0}$ exceeds $r_{0}$.

The above behavior of the random walk is closely related to the cutoff
phenomenon, which is defined in general as follows (see \cite{diaconis1996cutoff}).
Let $\left(P_{n}(x,y),X_{n}\right)$ be a series of Markov random
walks on a probability space $X_{n}$, and let $P_{n}^{t}\left(x,y\right)$
the $t$-step random walk on $X_{n}$. Let $f(n),g(n)$ be functions
such that $f(n)$ tends to infinity and $g(n)=o\left(f(n)\right)$
as $n\to\infty$. We say that the series $\left(P_{n}(x,y),X_{n}\right)$
exhibits a cutoff at time $f(n)$ with window of size $g(n)$, if
for every $1>\epsilon>0$, the time $t_{n}=\inf\left\{ t\mid\sup_{x_{0}}\n{P_{n}^{t}(x_{0},\cdot)-\pi_{n}}<\epsilon\right\} $
satisfies $t_{n}=f(n)+O_{\epsilon}\left(g\left(n\right)\right)$.
Determining whether a series of random walks exhibit a cutoff is a
fundamental problem (see \cite{diaconis1996cutoff}). Theorem \ref{thm2}
says that if a sequence of surfaces $X_{n}$ have injectivity radius
at least $r_{0}$ at every point of every surface then the random
walks on them exhibit a cutoff (and moreover the mixing time from
each point is the same and can be estimated explicitly).

\subsection*{Arithmetic Subgroups}

As said, Selberg's conjecture implies that the quotients $X$ of $\H$
by congruence subgroups of $SL_{2}\left(\mathbb{Z}\right)$ satisfy
the results of Theorem \ref{thm1} and Theorem \ref{thm2}. Using
the corrent knowledge, we can give slightly weaker statements (at
least for Theorem \ref{thm1}), which capture the essence of the result. 
\begin{thm}
\label{thm:SL2Z}Let $\Gamma=SL_{2}\left(\Z\right)$ or any cocompact
arithmetic lattice in $SL_{2}\left(\R\right)$, $X_{0}=\Gamma\backslash\H$
the corresponding quotient, $q\in\mathbb{N}$, $\Gamma\left(q\right)$
the principal congruence subgroup of $\Gamma$, $X_{q}=\Gamma\left(q\right)\backslash\H$
the corresponding quotient, and $\rho_{q}:X_{q}\to X_{0}$ be the
cover map.

Let $x_{0}^{(q)}\in X_{q}$ be a point such that its projection $\rho_{q}\left(x_{0}^{(q)}\right)$
to $X_{0}$ has injectivity radius at least a constant $r_{0}$. Then
for every $\epsilon_{0}>0$
\[
\mu\left(x\in X_{q}:d_{X_{q}}\left(x,x_{0}^{(q)}\right)\ge R_{X_{q}}\left(1+\epsilon_{0}\right)\right)/\mu\left(X_{q}\right)\to_{q\to\infty}0.
\]
\end{thm}

\subsection*{Methods of Proof}

The proofs of the three Theorems exploit the following proposition:
\begin{prop}
\label{prop:bounds on sepctrum}The surface $X$ is Ramanujan if and
only if for every $r\ge0$ the non-trivial spectrum of $A_{r}$ on
$L_{0}^{2}\left(X\right)=\left\{ f\in L^{2}\left(X\right):\int f(x)dx=0\right\} $
is bounded by $\left(r+1\right)e^{-r/2}$.
\end{prop}

A similar (generalized) proposition plays a crucial role in the work
of Harish-Chandra (see \cite[Theorem 3]{harish1958spherical}). Theorem
\ref{thm1} is actually a direct application of Proposition \ref{prop:bounds on sepctrum}.

The proof of Theorem \ref{thm2} combines Proposition \ref{prop:bounds on sepctrum}
with two other results. The first one, Proposition \ref{lem:Technical Lemma2},
says that after 3 steps the random walk measure $A_{r_{1}}^{3}\delta_{x_{0}}$
(respectively, the Brownian motion measure $B_{t_{0}}\delta_{x_{0}}$
at a fixed time $t_{0}>0$) is an $L^{2}$ function on $X$, with
a bounded $L^{2}$ norm depending only on the injectivity radius $r_{0}$.
The second result, Proposition \ref{cor:integral and CLT for discrete}
and Proposition \ref{prop:Integral and CLT for Brownian}, which is
well known for the Brownian motion, may be described as a concentration
of measure theorem for the rate of escape of the random walk $A_{r_{1}}^{k}$
(respectively $B_{t}$) on $\H$. Namely, we may write $A_{r_{1}}^{k}\cong\int_{r}f_{k}(r)A_{r}dr$
(respectively $B_{t}\cong\int_{r}g_{t}(r)A_{r}dr$), where most of
the measure $f_{k}(r)dr$ is concentrated at $\sim\alpha kr_{1}$
(respectively, $g_{t}(r)dr$ is concentrated at $\sim t$).

The proof of Theorem \ref{thm:SL2Z} depends on the following facts:
$\Gamma\left(q\right)$ is normal in $\Gamma$, there exists an absolute
lower bound on the smallest eigenvalue of $X_{q}$, on a careful general
analysis of the required bounds on the number of exeptional eigenvalues
of $X_{q}$, and on an $L_{p}$ generalization of Proposition \ref{prop:bounds on sepctrum}.
It is a beautiful result that the bound of the number of exceptional
eigenvalues that is required is exacly the ``elementary'' density
bound discussed in \cite{sarnak1991bounds}. The bound states the
number of eigenvalues of $X_{q}$ with corresponding matrix coefficients
not in $L^{p}$ for $p>2$ is $\ll_{\epsilon}\left[\Gamma:\Gamma\left(q\right)\right]^{2/p+\epsilon}$
(see also \cite{sarnak1990diophantine}). Note that the article \cite{sarnak1991bounds}
assumes cocompactness, which was removed in \cite{huntley1993density}
(stronger results for $SL_{2}\left(\Z\right)$ were also proven earlier
by different methods in \cite{iwaniec1985density,huxley1986exceptional})
. Theorem \ref{thm:SL2Z} also holds for $SL_{2}\left(\Z\right)$
for non-prime $q$, as a non-elementary bound on the smallest eigenvalue
was proven already by Selberg in \cite{selberg1965estimation}. See
the discussion in Section \ref{sec:Coverings} for full details.

This work is similar in spirit to the results of \cite{lubetzky2015cutoff},
and shows the general connection between the common distance and cutoff
phenomena in quotients of symmetric spaces (infinite regular trees
and the hyperbolic plane in these cases) and temperedness of representations
(or the Ramanujan conjecture). 

\subsection*{Open questions}

We expect that the results of this article can be extended to quotients
of higher dimensions, and also to other contexts (e.g. the action
of Hecke operators on $SL_{2}\left(Z\right)\backslash SL_{2}\left(\mathbb{R}\right)$
and its covers). Theorems analogous to Theorem \ref{thm1} for quotients
of $p$-adic Lie groups (i.e. Ramanujan complexes) are proven in \cite[Theorem 1.9]{kamber2016lpcomplex},
and \cite[Theorem 1.ii]{lubetzky2017random}. Theorem \ref{thm1}
is also closely related to the optimal covering properties of the
Golden-Gates of \cite{parzanchevski2017super}.

While we were unable to show it, we believe that it is possible to
prove in the notations of Theorem \ref{thm:SL2Z} that (at least for
$SL_{2}\left(\Z\right)$) there exists a constant $C>0$ such that
\[
\mu\left(x\in X_{q}:d_{X_{q}}\left(x,x_{0}^{(q)}\right)\ge R_{X_{q}}+C\ln\left(R_{X_{q}}\right)\right)/\mu\left(X_{q}\right)\to_{q\to\infty}0.
\]
Selberg's conjecture would give $C=2+\epsilon$, $\epsilon>0$ by
Theorem \ref{thm1}. The stronger density theorems of \cite{iwaniec1985density}
fall just a bit short of proving it. See Remark \ref{rem:covering remark}.

The following conjectures are natural continuous analogs of well known
combinatorial results, in the spirit of this article. Assume that
the lattice $\Gamma$ is a free group (for example, the principal
congruence subgroup $\Gamma=\Gamma(2)=\ker\left\{ PSL_{2}\left(\Z\right)\stackrel{\text{mod}}{\to}PSL_{2}\left(\Z/2\Z\right)\right\} $,
which is freely generated by $\left(\begin{array}{cc}
1 & 2\\
0 & 1
\end{array}\right)$ and $\left(\begin{array}{cc}
1 & 0\\
2 & 1
\end{array}\right)$). Then every onto homomorphism $\phi:\Gamma\to S_{n}$ defines an
index $n$ subgroup $\Gamma'\subset\Gamma$, by $\Gamma'=\left\{ \gamma\in\Gamma:\phi(\gamma)(1)=1\right\} $,
and every index $n$ subgroup of $\Gamma$ can be defined this way.
Since each homomorphism is defined using the generators, there is
a finite number of such homomorphism, and it defines a probability
measures on the index $n$ subgroups of $\Gamma$, or equivalently,
the $n$-covers of $X$.
\begin{conjecture}
\label{conj:comb conjectures}Assume that $\Gamma$ is a free group.

(1) For every $\epsilon>0$, the probability that every new eigenvalue
$\lambda$ of an $n$-cover $X'$ of $X$ satisfies $\lambda\ge1/4+\epsilon$
is $1-o(1)$.

An analogous statement for graphs is called Alon's conjecture, and
was proved in \cite{friedman2003proof}. 

(2) There exists a $2$-cover $X'$ of $X$, such that every new eigenvalue
$\lambda$ of $X'$ satisfies $\lambda\ge1/4$.

In the graph setting this statement is called Bilu-Linial's conjecture,
and was solved for the bipartite case in \cite{marcus2013interlacing}.
\end{conjecture}

See also \cite{brooks2004random}, where (in a slightly different
random model) weaker versions of (1) are proved.

\subsection*{Outline of the Article}

In Section \ref{sec:Preliminaries-and-Notation} we set notations
and discuss the harmonic analysis on $\H$, and its relation to the
operator $A_{r}$ and the Laplacian. We also prove a bound on the
$L^{2}$-spectrum of $A_{r}$ on $\H$. In Section \ref{sec:Proof-of-Theorem1}
we prove Theorem \ref{thm1}. In Section \ref{sec:The-Random-Walk}
we prove some versions of the central limit theorem for the random
walks. For the discrete random walk we reduce the problem to the standard
central limit theorem. For the Brownian motion this result is well
known. In Section \ref{sec:DeltaToL2} we prove that after a short
time the random walks sends the delta measure on a point to a bounded
$L^{2}$-function. In Section \ref{sec:Proof-of-Theorem2} we prove
Theorem \ref{thm2}. 

In the rest of the article we prove a generalized version of Theorem
\ref{thm:SL2Z}. In Section \ref{sec:Lp-expanders} we generalize
the bounds for spectra and Proposition \ref{prop:bounds on sepctrum}
to the non-Ramanujan case. We also give a weak version of Theorem
\ref{thm1}, which depends on the smallest non-trivial eigenvalue
of the Laplacian. In Section \ref{sec:Coverings} we discuss covers
of a fixed quotient $X_{0}$, and in particular normal covers. The
requirement on the spectrum of normal covers is stated somewhat abstractly
in Theorem \ref{thm:Normal Coverings Theorem}. We then discuss density
theorems and known results about them, and show that the density theorems
satisfy the requirements of Theorem \ref{thm:Normal Coverings Theorem},
thus proving Corollary \ref{cor:Density Theorem Corollary}, which
implies Theorem \ref{thm:SL2Z}.

We also have two appendices. In the first appendix, Section \ref{sec:Concentration of distance},
we prove that for any fixed $x_{0}\in X$ there exists a distance
$R_{x_{0},X}$ such that the distances from $x_{0}$ to other vertices
is concentrated around $R_{x_{0},X}$ in a window of a constant size,
where the constant depends on the smallest non-trivial eigenvalue
of the Laplacian. Theorem \ref{thm1} implies that if $X$ is Ramanujan
and $x_{0}$ has a lower bound on its injectivity radius, then $R_{X}\le R_{X,x_{0}}\le R_{X}+\left(2+\epsilon\right)\ln R_{X}$.
The proof involves some interesting isoperimetric inequalities.

In the second appendix, Section \ref{sec:Cutoof and Comparison-with-the-flat},
we show that the Gaussian random walks on the flat surfaces $\left(a\mathbb{Z}^{2}\right)\backslash\mathbb{\mathbb{R}}^{2}$,
$a\to\infty$ do not exhibit a cutoff.

\subsection*{Acknowledgments}

We are grateful to Elon Lindenstrauss, Alex Lubotzky, Shachar Mozes
and Józef Dodziuk for fruitful discussions. The first author is supported
by the ERC grant 336283. This work is part of the Ph.D. thesis of
the second author at the Hebrew University of Jerusalem. A large part
of this work was carried out in the café Bread\&Co in Tel-Aviv, to
which we are thankful for its coffee and hospitality.

\section{\label{sec:Preliminaries-and-Notation}Preliminaries}

\subsection*{The hyperbolic plane}

There are several models for the hyperbolic plane $\H$ of constant
curvature $-1$, and we stick to the upper half-plane model. That
is the complex half-plane $\{z\in\mathbb{C}\mid\text{Im}(z)>0\}$
endowed with the metric $ds^{2}=dz^{2}/\text{(Im}(z))^{2}$. For $z=x+iy,\:z'=x'+iy'\in\mathbb{H}$
the distance $d(z,z')$ between them is 
\[
d\left(\left(x,y\right),\left(x',y'\right)\right)=\text{acosh\ensuremath{\left(1+\frac{\left(x'-x\right)^{2}+\left(y'-y\right)^{2}}{2yy'}\right).}}
\]
The group $G=PSL_{2}(\R)$ acts on $\mathbb{H}$ by Mobius transformations,
i.e., 
\[
\left(\begin{array}{cc}
a & b\\
c & d
\end{array}\right)\cdot z=\frac{az+b}{cz+d},
\]
and constitutes the group of orientation preserving isometries of
$\mathbb{H}$. It also acts transitively on the points of $\mathbb{H}$,
with the subgroup $K=PSO_{2}(\mathbb{R})\subset G$ being the stabilizer
of the point $i$, to which we refer as the origin of $\H$. The subgroup
$K$ acts on $\H$ by rotations around $i$. The plane $\mathbb{H}$
can be identified with the quotient $G/K$, and in particular, the
circle of radius $r$ around $i$ identifies with the double coset
$K\left(\begin{array}{cc}
e^{r/2} & 0\\
0 & e^{-r/2}
\end{array}\right)K$. The Haar measure on $G$ which is normalized so that the measure
of $K$ is equal to $1$ agrees with the standard measure $\mu$ on
$\H$.

\subsection*{Harmonic analysis on $\mathbb{H}$}

For $f\in L^{1}\left(\H\right)$, its Helgason-Fourier transform $\widehat{f}(s,k)\in C\left(\C\times K\right)$,
is defined for $s\in\mathbb{C}$ and $k\in K=PSO_{2}(\mathbb{R})$
as 
\[
\widehat{f}(s,k)=\intop_{\mathbb{H}}f(z)\overline{\left(\text{Im}(kz)\right)^{\frac{1}{2}+is}}dz.
\]
In the case when $f$ is $K$-invariant, i.e., $f(kz)=f(z)$ for all
$z\in\H$ and $k\in K$, its transform is independent of $k$ and
can be written with the help of the spherical functions. For every
$s\in\C,$ the corresponding spherical function is a $K$-invariant
function defined as 
\[
\varphi_{\frac{1}{2}+is}(z)=\int_{K}\overline{\left(\text{Im}(kz\right)^{\frac{1}{2}+is-2}}dk.
\]
Since $\varphi_{\frac{1}{2}+is}$ is $K$-invariant, it depends solely
on the distance from a point to the origin $i$, and can be written
as
\[
\varphi_{\frac{1}{2}+is}(z)=\varphi_{\frac{1}{2}+is}(ke^{-r}i)=P_{-\frac{1}{2}+is}(\cosh r),
\]
where $k\in K$, $r\in\R_{\ge0}$ is the distance from $z$ to $i,$
and $P_{s}(r)$ is the Legendre function of the first kind. Then the
Helgason-Fourier transform of a $K$-invariant function $f$ reads
as
\[
\widehat{f}(s)=\intop_{\mathbb{H}}f(z)\varphi_{\frac{1}{2}+is}(z)dz=\intop_{0}^{\infty}f(e^{-r}i)P_{-\frac{1}{2}+is}(\cosh r)\sinh rdr.
\]

For two functions $f_{1},f_{2}\in L^{1}\left(\H\right)$, their convolution
is defined as
\[
f_{1}\ast f_{2}(z)=\int_{G}f_{1}(gi)f_{2}(g^{-1}z)dg.
\]
We exploit of the following properties of the Helgason-Fourier transform
on $\H$. For an extensive presentation of the theory, see \cite{Helgason2000,Terras2013}.
\begin{prop}
\label{prop:Plancherel-Helgason}(\cite[Theorem 3.2.3]{Terras2013})
\begin{enumerate}
\item (Plancherel Formula) The map $f\to\widehat{f}$ extends to an isometry
of $L^{2}\left(\mathbb{H},d\mu\right)$ with $L^{2}\left(\mathbb{R}\times K,\frac{1}{4\pi}s\tanh\pi s\,dsdk\right)$,
where the $K$ is identified with the interval $[0,1)$.
\item (Convolution property) For $f,g\in L^{1}(\mathbb{H})$, where $g$
is $K$-invariant, 
\[
\widehat{f\ast g}=\widehat{f}\cdot\widehat{g},
\]
where $\ast$ stands for convolution, and $\cdot$ for pointwise multiplication.
\end{enumerate}
\end{prop}

The Helgason-Fourier transform can be extended to compactly supported
measures on $\H$. Namely, for such a measure $\nu$, its transform
$\widehat{f}(s,k)\in C\left(\C\times K\right)$, is defined for $s\in\mathbb{C}$
and $k\in K=PSO_{2}(\mathbb{R})$ as 
\[
\widehat{\nu}(s,k)=\intop_{\mathbb{H}}\overline{\left(\text{Im}(k(z))\right)^{\frac{1}{2}+is}}d\nu,
\]
and, if the measure is $K$-invariant, its transform is independent
of $k$, and can be written as
\[
\widehat{\rho}(s)=\intop_{\H}\varphi_{\frac{1}{2}+is}(z)d\rho.
\]

We will need the following claim, which follows from Theorem \ref{prop:Plancherel-Helgason}. 
\begin{cor}
\label{cor:L2 fourier transform}Let $\nu$ be a compactly supported
distribution on $\H$, and assume that $\widehat{\nu}\in L^{2}\left(\mathbb{R}\times K,\frac{1}{4\pi}s\tanh\pi s\,dtdk\right)$.
Then $\nu$ can be represented as an $L^{2}$ function on $\H$, i.e.
there exists $f_{\nu}\in L^{2}\left(\H\right)$ such that for every
$f\in C_{c}\left(\H\right)$, $\nu\left(f\right)=\int f_{\nu}(x)f(x)dx$.
\end{cor}

\subsection*{The averaging operator $A_{r}$}

For $r>0,$ let $A_{r}$ denote the operator on $C(\mathbb{H})$ which
averages a function over a circle of radius $r,$ i.e., for a function
$f\in C(\mathbb{H})$ and $z\in\mathbb{H}$
\[
\left(A_{r}f\right)(z)=\intop_{K}f\left(k\left(\begin{array}{cc}
e^{r/2} & 0\\
0 & e^{-r/2}
\end{array}\right)z\right)dk.
\]
The operator $A_{r}$ is bounded and self adjoint with respect to
the $L^{2}$-norm on $L^{2}\left(\H\right)\cap C\left(\H\right)$,
so it extends to an operator $A_{r}:L^{2}\left(\H\right)\to L^{2}\left(\H\right)$,
which is also self adjoint. By duality, we may also extend $A_{r}$
to an operator on compactly supported measures on $\H$. Note that
the operator $A_{r}$ can written as a convolution with a uniform
$K$-invariant probability measure $\delta_{S_{r}}$ supported on
the double coset $K\left(\begin{array}{cc}
e^{r/2} & 0\\
0 & e^{-r/2}
\end{array}\right)K$. 

Note that the Laplace-Beltrami operator $\Delta=-y^{2}\left(\frac{\partial^{2}}{\partial x^{2}}+\frac{\partial^{2}}{\partial y^{2}}\right)$
can be written on $C^{\infty}\left(\H\right)$ as the limit
\[
\Delta=(-2)\lim_{r\to0}\frac{1}{r^{2}}\left(I-A_{r}\right),
\]
where $I$ stands for the identity operator. However, we are mainly
concerned with the behavior of $A_{r}$ when $r$ is either fixed
or approaches infinity.

The spherical functions $\varphi_{\frac{1}{2}+is}$ on $\H$ are eigenfunctions
of $\Delta$ and of $A_{r}$ for every $r>0$, namely,
\begin{align*}
\Delta\varphi_{\frac{1}{2}+is} & =\left(\frac{1}{4}+s^{2}\right)\varphi_{\frac{1}{2}+is}\\
A_{r}\varphi_{\frac{1}{2}+is} & =\varphi_{\frac{1}{2}+is}(e^{-r}i)\cdot\varphi_{\frac{1}{2}+is}.
\end{align*}
In particular, it follows from Proposition \ref{prop:Plancherel-Helgason}
that the $L^{2}$-spectrum of $\Delta$ on $\H$ is $\left[\frac{1}{4},\infty\right)$
and the $L^{2}$-spectrum of $A_{r}$ on $\H$ is the set $\left\{ \varphi_{\frac{1}{2}+is}(e^{-r}i)\mid s\in\R\right\} =\left\{ P_{-\frac{1}{2}+is}(\cosh r)\mid s\in\R\right\} $.

\subsection*{Spectrum on the Quotients and the Ramanujan condition.}

Consider the actions of $A_{r}$ and of $\Delta$ on a dense subspace
of $L_{0}^{2}\left(\Gamma\backslash\mathbb{H}\right)=\left\{ f\in L^{2}\left(\Gamma\backslash\mathbb{H}\right):\int f=0\right\} $,
where $\Gamma\subseteq PSL_{2}(\mathbb{R})$ is a lattice. In both
cases the spectrum is not necessarily discrete, but is parameterized
by the unitary dual parameter $\frac{1}{2}+is\in\mathbb{C}$. Namely,
if $\frac{1}{2}+is\in\mathbb{C}$ appears in the unitary dual of $X=\Gamma\backslash\H$,
then $P_{-\frac{1}{2}+is}(\cosh r)$ is in the spectrum of $A_{r}$
and $\frac{1}{4}+s^{2}$ is an eigenvalue of the Laplacian. It is
well known that, in general, the unitary dual is contained in the
set $\left\{ \frac{1}{2}+is\mid s\in\R\right\} \cup\left\{ \frac{1}{2}+is\mid is\in\left(-\frac{1}{2},\frac{1}{2}\right)\right\} \cup\left\{ 0,1\right\} $,
where the first set is called the principal series, the second one
is called the complementary series, and $\{0,1\}$ is called trivial.
The trivial part corresponds to the constant function on $X$. A quotient
$X=\Gamma\backslash\mathbb{H}$ is called Ramanujan if its non-trivial
unitary dual is contained solely in $\left\{ \frac{1}{2}+is\mid s\in\R\right\} $.
Equivalently, $X$ is Ramanujan iff all the non-trivial eigenvalues
of the Laplacian are greater or equal to $\frac{1}{4}$. 

Another equivalent condition of being a Ramanujan quotient is that
the subrepresentation of $G$ on $L_{0}^{2}\left(\Gamma\backslash G\right)$
generated by its $K=PSO_{2}\left(\R\right)$ fixed vectors, is tempered.
This statement can also be stated as follows. Every function $f$
on $X$ can be lifted to a $\Gamma$-invariant function $\tilde{f}$
on $\H$. Then $X$ is Ramanujan iff for every $f,f'\in L_{0}^{2}\left(X\right)$,
and for every $\epsilon>0$,
\[
\int_{G}\left|\left\langle \tilde{f},g\tilde{f}'\right\rangle \right|^{2+\epsilon}dg<\infty.
\]

\subsection*{Harish-Chandra Bounds.}
\begin{prop}
The spectrum of $A_{r}$ on $L^{2}\left(\H\right)$ is bounded by
$\left(r+1\right)e^{-r/2}$.
\end{prop}

\begin{proof}
The $L^{2}-$spectrum is composed of eigenvalues of $A_{r}$ on the
principal series spherical functions, and hence is equal to $\{P_{-\frac{1}{2}+is}(\cosh r)\}_{s\in\mathbb{R}}$.
Alternatively, it can be written as the range of the function $\phi_{\frac{1}{2}+is}(r)=\frac{\sqrt{2}}{\pi}r\int_{0}^{1}\frac{\cos\left(srx\right)}{\sqrt{\cosh r-\cosh rx}}dx$,
for $s\in\mathbb{R}$ (\cite[Lemma 7]{decorte2017lower}, or \cite[Exercise 3.2.28]{Terras2013},
). Since $\cosh r-\cosh\left(rx\right)\ge\left(\cosh r-1\right)\left(1-x^{2}\right)$,
for $0\leq x\leq1,$ (which follows from the Taylor expansion of $\cosh$),
the following inequalities hold%
\begin{align*}
\left|\phi_{\frac{1}{2}+is}(r)\right| & =\frac{\sqrt{2}}{\pi}r\left|\int_{0}^{1}\frac{\cos\left(srx\right)}{\sqrt{\cosh r-\cosh rx}}dx\right|\le\frac{\sqrt{2}}{\pi}r\frac{1}{\sqrt{\cosh r-1}}\int_{0}^{1}\frac{1}{\sqrt{1-x^{2}}}dx\\
 & \le\frac{\sqrt{2}}{\pi}r\frac{1}{\sqrt{\cosh r-1}}\int_{0}^{1}\frac{1}{\sqrt{1-x^{2}}}dx=\frac{1}{\sqrt{2}}r\left(\cosh r-1\right)^{-1/2}=\frac{r}{2}\left(\sinh\frac{r}{2}\right)^{-1}\le(r+1)e^{-r/2}.
\end{align*}
\end{proof}
\begin{cor}
If $X$ is Ramanujan then the norm of $A_{r}$ on $L_{0}^{2}\left(X\right)$
is bounded by $(r+1)e^{-r/2}$ .
\end{cor}

The inverse direction can be proven in a similar way, by analyzing
the complementary series. Let us present a more conceptual proof of
it: 
\begin{prop}
\label{prop:ArBoundsImpliesRamanujan}If for every $r\ge0$ the norm
of $A_{r}$ on $L_{0}^{2}\left(X\right)$ is bounded by $(r+1)e^{-r/2}$
then $X$ is Ramanujan.
\end{prop}

\begin{proof}
Recall that the condition on the Laplacian is equivalent to the fact
that the subrepresentation of $G$ on $L_{0}^{2}\left(\Gamma\backslash G\right)$
with $K=PSO_{2}\left(\R\right)$ fixed vectors, is tempered. Consider
the Cartan decomposition $G=\cup_{r\ge0}K\left(\begin{array}{cc}
e^{r} & 0\\
0 & e^{-r}
\end{array}\right)K$, which corresponds to the polar coordinates in $\H.$ The metric
on the group in this coordinates reads as $dg=\sinh rdk'dkdr$. Let
$f,f'\in L_{0}^{2}\left(X\right)$ and let $\tilde{f},\tilde{f}'\in L_{0}^{2}\left(\Gamma\backslash G\right)$
be their lifts. Then 
\begin{align*}
\int_{G}\left|\left\langle \tilde{f},g\tilde{f}'\right\rangle \right|^{2+\epsilon}dg & =\int_{r\ge0}\sinh r\left|\left\langle \tilde{f},\left(\begin{array}{cc}
e^{r} & 0\\
0 & e^{-r}
\end{array}\right)\tilde{f}'\right\rangle \right|^{2+\epsilon}dr=\\
 & =\int_{r\ge0}\sinh r\left|\left\langle f,A_{r}f'\right\rangle \right|^{2+\epsilon}dr
\end{align*}
Using the fact that for $r$ large $\sinh r\asymp e^{r}$, we see
that the Ramanujan condition is equivalent to the condition:
\begin{itemize}
\item For every $f,f'\in L_{0}^{2}\left(X\right)$ and for every $\epsilon>0$,$\int_{r\ge0}e^{r}\left|\left\langle f,A_{r}f'\right\rangle \right|^{2+\epsilon}dr<\infty$
\end{itemize}
If the inequality holds then $\left|\left\langle f,A_{r}f'\right\rangle \right|\le(r+1)e^{-r/2}\left|\left\langle f,f'\right\rangle \right|$
for every positive $r$, so 
\begin{align*}
\int_{r\ge0}e^{r}\left|\left\langle f,A_{r}f'\right\rangle \right|^{2+\epsilon}dr & \le\int_{r\ge0}e^{r}e^{\left(-1-\epsilon/2\right)r}(r+1)^{2+\epsilon}\left|\left\langle f,f'\right\rangle \right|^{2+\epsilon}dr=\\
 & =\left|\left\langle f,f'\right\rangle \right|^{2+\epsilon}\int_{r\ge0}(r+1)^{2+\epsilon}e^{-\epsilon r}dr<\infty,
\end{align*}
and the proposition follows.%
\end{proof}

\section{\label{sec:Proof-of-Theorem1}Proof of Theorem \ref{thm1}}
\begin{proof}
of Theorem \ref{thm1}. Let $r\le R_{X}-\gamma\ln\left(R_{X}\right)$.
The measure of $Y_{<}=\left\{ x\in X:d\left(x,x_{0}\right)<r\right\} $
is at most the volume of the ball of radius $r$ in the hyperbolic
plane, i.e.,
\begin{align}
\mu\left(Y_{<}\right) & \le\mu\left(B_{r}\right)\ll e^{r}=e^{R_{X}}e^{-\gamma\ln\left(R_{X}\right)}\ll\mu\left(X\right)R_{X}^{-\gamma},\label{eq:lower bound}
\end{align}
which implies the lower bound of the theorem (note that we assume
that $\mu\left(X\right)\asymp e^{R_{X}}$ since $R_{X}\ge1$).

Now let $r'=R_{X}+\gamma\ln\left(R_{X}\right)-r_{0}$, and $Y_{>}=\left\{ x\in X:d_{X}\left(y,x_{0}\right)>r'\right\} $.
Let $b_{x_{0},r_{0}}$ be the characteristic function of $B_{x_{0}}\left(r_{0}\right)\subset X$,
normalized as follows: 
\[
b_{x_{0},r_{0}}(x)=\begin{cases}
1/\mu\left(B_{r_{0}}\right), & x\in B_{x_{0}}(r_{0});\\
0, & x\not\in B_{x_{0}}(r_{0}).
\end{cases}
\]
It is well defined since $x_{0}$ has injectivity radius at least
$r_{0}$. Then $Y_{>}\subset Z$ where $Z=\left\{ x\in X:\;A_{r^{\prime}}b_{x_{0},r_{0}}\left(x\right)=0\right\} $.
Denote $\pi\in L^{2}(X)$ the constant function with $\pi(x)=1/\mu(X)$
for every $x\in X$. For every point $x\in Z$, one has $\left|\left(A_{r^{\prime}}b_{x_{0},r_{0}}-\pi\right)(x)\right|=\pi(x)=\frac{1}{\mu\left(X\right)}$,
so $\mu\left(Z\right)\mu^{-2}\left(X\right)\le\nn{A_{r^{\prime}}b_{x_{0},r_{0}}-\pi}^{2}$.
Therefore $\mu\left(Y_{>}\right)\le\mu\left(Z\right)\le\mu^{2}\left(X\right)\nn{A_{r_{n}^{\prime}}b_{x_{0},r_{0}}-\pi}^{2}$.

Since $b_{x_{0},r_{0}}-\pi\perp\pi$ in the space $L^{2}(X)$, it
holds that 
\[
\nn{b_{x_{0},r_{0}}-\pi}\le\nn{b_{x_{0},r_{0}}}=\mu\left(B_{r_{0}}\right)^{-1/2}\ll_{r_{0}}1.
\]
The bounds on the norm of $A_{r'}$ of Proposition \ref{prop:bounds on sepctrum}
imply the following inequality
\begin{align*}
\nn{A_{r^{\prime}}b_{x_{0},r_{0}}-\pi} & =\nn{A_{r^{\prime}}\left(b_{x_{0},r_{0}}-\pi\right)}\le\left(r'+1\right)e^{-r^{\prime}/2}\nn{b_{x_{0},r_{0}}-\pi}\\
 & \ll_{r_{0}}\left(\frac{R_{X}+\gamma\ln\left(R_{X}\right)-r_{0}+1}{R_{X}}\right)R_{X}e^{-\frac{1}{2}R_{X}-\frac{1}{2}\gamma\ln\left(R_{X}\right)+\frac{1}{2}r_{0}}\\
 & \ll_{r_{0}}\left(1+\gamma\right)e^{-R_{X}/2}R_{X}^{1-\gamma/2}\ll\left(1+\gamma\right)\mu\left(X\right)^{-1/2}R_{X}^{1-\gamma/2}.
\end{align*}
And he following inequality completes the proof
\[
\mu\left(Y_{>}\right)\le\mu^{2}\left(X\right)\nn{A_{r^{\prime}}b_{x_{0},r_{0}}-\pi}^{2}\ll_{r_{0}}\mu\left(X\right)\left(1+\gamma^{2}\right)R_{X}^{2-\gamma}.
\]
\end{proof}

\section{\label{sec:The-Random-Walk}Deviations of the Random Walk}

Let $r_{1}>0$ be fixed. Consider the random walk on $\H$, emanating
from $z_{0}=i$ and having $z_{k+1}$ equidistributed on the sphere
of radius $r_{1}$ around $z_{k}$ . In other words, $z_{k}$ distributes
according to the measure $A_{r_{1}}^{k}\delta_{z_{0}}$, where $\delta_{z_{0}}$
is the Dirac delta-measure at $z_{0}$. Write $z_{k}=x_{k}+y_{k}i$
for $k\in\N\cup\{0\}$.

Recall that in the upper half-plane model, the points at infinity
of $\H$ are $\R\cup\left\{ \infty\right\} $. In the following lemma
we show that the random walk $A_{r_{1}}^{k}\delta_{z_{0}}$ moves
away from $\infty$ at a constant speed.
\begin{lem}
Let $f:\left[0,\pi\right]\to\left[-1,1\right]$ be the function $f(\theta)=-\frac{1}{r_{1}}\ln\left(e^{r_{1}}\cos^{2}\theta+e^{-r_{1}}\sin^{2}\theta\right)$.
Let $m$ be the uniform probability measure on $\left[0,\pi\right]$
and let $\nu=f^{*}m$ be the induced probability measure on $\left[-1,1\right]$
(i.e. for $A\subset\left[-1,1\right]$, $\nu(A)=m\left(f^{-1}\left(A\right)\right)$).
Then $\frac{1}{r_{1}}\ln(y_{k})$ distributes according to $\nu*\nu*...*v$
($k$ times). In other words, $\ln\left(y_{k}\right)=\ln\left(y_{k-1}\right)+r_{1}Y$,
where $Y$ is a random variable, independent from $y_{k-1}$, that
distributes according to $\nu$.
\end{lem}

\begin{proof}
One should show that for a given point $z\in\H$, the logarithm of
the imaginary part of the measure $A_{r_{1}}\delta_{z}$ is distributed
according to $\ln\left(\text{Im}z'\right)=\ln\left(\text{Im}z\right)-\ln\left(e^{r_{1}}\cos^{2}\theta+e^{-r_{1}}\sin^{2}\theta\right)$,
for $0\le\theta\le\pi$ equidistributed.

In the case of $z=i$, the sphere of radius $r_{1}$ around $z$ can
be parameterized as 
\[
S_{r_{1}}\left(i\right)=\left\{ \left(\begin{array}{cc}
e^{r_{1}/2}\sin\theta & e^{-r_{1}/2}\cos\theta\\
-e^{r_{1}/2}\cos\theta & e^{-r_{1}/2}\sin\theta
\end{array}\right)\cdot i=\frac{i+\sin\theta\cos\theta\left(e^{-r_{1}}-e^{r_{1}}\right)}{e^{2r_{1}}\cos^{2}\theta+\sin^{2}\theta}\mid0\le\theta<\pi\right\} .
\]
The logarithm of the imaginary part $y'$ of $z'=A_{r_{1}}z$ distributes
according to $\ln(y')=-\ln\left(e^{r_{1}}\cos^{2}\theta+e^{-r_{1}}\sin^{2}\theta\right)$.

To prove the claim for points other than $z=i$, notice that an isometry
$g$ of $\H$ maps $A_{r_{1}}\delta_{z}$ to $A_{r_{1}}\delta_{g\cdot z}$.
Since the action of $\left(\begin{array}{cc}
1 & s\\
0 & 1
\end{array}\right)$, $s\in\R$ does not change the imaginary coordinate of a point and
maps a point $z\in\H$ to $z'=z+s$, if the claim holds for $z$,
it also holds for $z+s$. Similarly, the action of $\left(\begin{array}{cc}
e^{t/2} & 0\\
0 & e^{-t/2}
\end{array}\right)$, $t\in\R$, maps a point $z$ to $z'=e^{t}z$, and, in particular,
multiplies its imaginary coordinate by $e^{t}$, hence if the claim
is true for $z$, it is true for $e^{t}z$ as well. Therefore it is
holds for every point $z\in\H$.
\end{proof}
\begin{cor}
\label{cor:Central Limit thereom for iwasawa}The random variables
$\sqrt{k}^{-1}\left(k^{-1}r_{1}\ln(y_{k})+\alpha_{r_{1}}\right)$
converges in distribution to the normal distribution $N(0,\sigma_{r_{1}}^{2})$
, where
\begin{align*}
\alpha_{r_{1}} & =\frac{1}{\pi r_{1}}\int_{0}^{\pi}\ln\left(e^{r_{1}}\cos^{2}\theta+e^{-r_{1}}\sin^{2}\theta\right)d\theta,\\
\sigma_{r_{1}}^{2} & =\frac{1}{\pi}\int_{0}^{\pi}\left(\frac{1}{r_{1}}\ln\left(e^{r_{1}}\cos^{2}\theta+e^{-r_{1}}\sin^{2}\theta\right)-\alpha_{r_{1}}\right)^{2}d\theta.
\end{align*}
Also, these numbers satisfies $0<\alpha_{r_{1}}<1$ and that $\sigma_{r_{1}}^{2}\le4$. 

Moreover, the Hoeffding inequality holds: there exist $c>0$ such
that for every $\lambda\ge0$ and $k\ge0$ 
\[
\text{Pr}\left(\left|\ln(y_{k})+\alpha_{r}kr_{1}\right|\ge\lambda r_{1}\sqrt{k}\right)\ll e^{-c\lambda^{2}}.
\]
\end{cor}

\begin{proof}
The statement is a direct application of the central limit theorem
and Hoeffding's inequality for independent bounded random variables.
The expectancy is equal to $\alpha_{r_{1}}$ and the variance is equal
to $\sigma_{r_{1}}$. The fact that $0<\alpha_{r_{1}}<1$ follows
from the fact that logarithm is a concave function.
\end{proof}
The random walk operator $A_{r_{1}}$ commutes with the action by
isometries on $\H$. The stabilizer of $i$ acts transitively on the
points at infinity of $\H$. Therefore, just as the random walk $A_{r_{1}}^{k}\delta_{z_{0}}$
moves away from $\infty$, it moves away from any other point at infinity. 
\begin{cor}
\label{cor:CLT for any point}Let $g\in G$ be an isometry of $\H$
fixing $i$, then Corollary \ref{cor:Central Limit thereom for iwasawa}
holds if we replace $y_{k}=\text{Im}z_{k}$ by $\text{Im}\left(g\cdot z_{k}\right)$,
i.e., $\sqrt{k}^{-1}\left(k^{-1}r_{1}\ln\left(\text{Im}\left(g\cdot z_{k}\right)\right)+\alpha_{r_{1}}\right)$
converges in distribution to the normal distribution $N(0,\sigma_{r_{1}}^{2})$
with $\alpha_{1}$ and $\sigma_{1}^{2}$ as in Corollary \ref{cor:Central Limit thereom for iwasawa}.
\end{cor}

In the following Lemma, we make a particular use of the above Corollary
for the isometry $g:z\mapsto-1/z$.
\begin{lem}
\label{lem:bounding x}There exists $c>0$ such that $\text{Pr}\left(x_{k}^{2}\ge\exp\left(\lambda r_{1}k^{1/2}\right)\right)\ll e^{-c\lambda^{2}}$
for all $\lambda>0$ and $k\ge0$.
\end{lem}

\begin{proof}
By Corollary \ref{cor:Central Limit thereom for iwasawa} there exists
$c_{0}>0$ such that 
\[
\text{Pr}\left(\left|\ln(y_{k})+\alpha_{r_{1}}kr_{1}\right|\ge\lambda r_{1}\sqrt{k}\right)\ll e^{-c_{0}\lambda^{2}}.
\]
By Corollary \ref{cor:CLT for any point} applied for $-\text{Im}z_{k}^{-1}=\frac{y_{k}}{x_{k}^{2}+y_{k}^{2}}$,
there exists $c_{1}>0$ such that 
\[
\text{Pr}\left(\left|\ln\left(\frac{y_{k}}{x_{k}^{2}+y_{k}^{2}}\right)+r_{1}\alpha_{r_{1}}k\right|\ge\lambda r_{1}\sqrt{k}\right)\ll e^{-c_{1}\lambda^{2}},
\]
and hence
\[
\text{Pr}\left(\left|\ln\left(x_{k}^{2}+y_{k}^{2}\right)\right|\ge2r_{1}\lambda\sqrt{k}\right)\ll e^{-c_{0}\lambda^{2}}+e^{-c_{1}\lambda^{2}}.
\]
Therefore there exists $c>0$ such that 
\[
\text{Pr}\left(x_{k}^{2}\ge\exp\left(r_{1}\lambda\sqrt{k}\right)\right)\le\text{Pr}\left(x_{k}^{2}+y_{k}^{2}\ge\exp\left(r_{1}\lambda\sqrt{k}\right)\right)\ll e^{-c\lambda^{2}}.
\]
\end{proof}
\begin{cor}
\label{cor:Central Limit thereom for distance} Let $z_{k}\propto A_{r_{1}}^{k}\delta_{z_{0}}$.
Then there exists $c=c(r_{1})>0$, such that for every $k\ge0$ and
$\lambda\ge0$
\begin{equation}
\text{Pr}\left(\left|d\left(z_{k},z_{0}\right)-\alpha_{r_{1}}r_{1}k\right|\ge\lambda\sqrt{k}\right)\ll_{r_{1}}e^{-c\lambda^{2}}.\label{eq:C,c depend on r1}
\end{equation}
\end{cor}

\begin{proof}
Let us start by proving that there exists $c>0$, such that for $k\ge0$,
$\lambda\ge0$,
\begin{equation}
\text{Pr}\left(\left|d\left(z_{k},z_{0}\right)-\alpha_{r_{1}}r_{1}k\right|\ge1+\lambda r_{1}\sqrt{k}\right)\ll e^{-c\lambda^{2}}\label{eq:CLT dix pf}
\end{equation}
For any point $z=x+iy\in\H$, the triangle inequality implies that
\begin{align*}
\left|d\left(z,i\right)-d\left(z,x+i\right)\right| & \le d\left(x+i,i\right)=\text{acosh}\left(1+\frac{x^{2}}{2}\right)\\
 & \le\max\left\{ 1,1+10\ln\left(x^{2}\right)\right\} .
\end{align*}
Hence by Lemma \ref{lem:bounding x} there exists $c_{0}>0$, such
that
\begin{equation}
\text{Pr}\left(\left|d\left(z,i\right)-d\left(z,x+i\right)\right|\le1+\lambda r_{1}\sqrt{k}\right)\ll e^{-c_{0}\lambda^{2}}.\label{eq:CLT_pf1}
\end{equation}
And by Corollary \ref{cor:Central Limit thereom for iwasawa} there
exists $c_{1}>0$, such that
\[
\text{Pr}\left(\left|\ln y+\alpha_{r_{1}}kr_{1}\right|\ge\lambda r_{1}\sqrt{k}\right)\ll e^{-c_{1}\lambda^{2}}.
\]
Since $\left|d\left(z,x+i\right)-\alpha_{r_{1}}kr_{1}\right|=\left|\left|\ln y\right|-\alpha_{r_{1}}kr_{1}\right|$,
if $\left|\left|\ln y\right|-\alpha_{r_{1}}kr_{1}\right|\ge\lambda r_{1}\sqrt{k}$
then also $\left|\ln y+\alpha_{r_{1}}kr_{1}\right|\ge\lambda r_{1}\sqrt{k}$,
and

\begin{align}
\text{Pr}\left(\left|d\left(z,x+i\right)-\alpha_{r_{1}}kr_{1}\right|\ge\lambda r_{1}\sqrt{k}\right) & \ll e^{-c_{1}\lambda^{2}},\label{eq:CLT_pf2}
\end{align}
which completes the proof.

Equation \ref{eq:C,c depend on r1} follows from Equation \ref{eq:CLT dix pf},
as for $\lambda\ge r_{1}^{-1}$ and $k>0$ $1+\lambda r_{1}\sqrt{k}\le2\lambda r_{1}\sqrt{k}$,
and we can choose $c'(r_{1})=c/r_{1}2$ and choose the constant of
$\ll_{r_{1}}$ in such a way that \ref{eq:C,c depend on r1} holds
for $\lambda\le r_{1}^{-1}$.
\end{proof}
\begin{rem}
One cannot hope to change $\left|d\left(z_{k},z_{0}\right)-\alpha_{r_{1}}kr_{1}\right|\ge1+\lambda r_{1}\sqrt{k}$
to $\left|d\left(z_{k},z_{0}\right)-\alpha_{r_{1}}kr_{1}\right|\ge\lambda r_{1}\sqrt{k}$
in the theorem without assuming dependency on $r_{1}$, since for
$r_{1}\to0$, $k\to\infty$ and $kr_{1}\to0$ the random walk behaves
like the distance $r_{1}$ random walk in $\R^{2}$, and in particular
it will not diverge at a constant speed. 
\end{rem}

Note that for $f\in L^{2}\left(\H\right)$ ($f\in L^{2}\left(X\right)$,
resp.), and for $x\in\H$ ($x\in X$, resp.), the following equality
holds 
\[
A_{r_{1}}^{k}f(x)=\int_{0}^{kr_{1}}\left(A_{r}f\right)(x)dm_{k}^{r_{1}}\left(r\right),
\]
for some probability measure $m_{k}^{r_{1}}$ supported on $\left[0,kr_{1}\right]$
and $k\in\N$. 
\begin{cor}
\label{cor:integral and CLT for discrete}There exists $c=c(r_{1})>0$
such that for every $k\ge0$
\[
\intop_{r:\left|r-kr_{1}\alpha_{r_{1}}\right|\le\lambda r_{1}\sqrt{k}}dm_{k}^{r_{1}}\left(r\right)\ll_{r_{1}}\exp\left(-c\lambda^{2}\right).
\]
\end{cor}

\begin{proof}
Follows directly from Corollary \ref{cor:Central Limit thereom for distance}.
\end{proof}
In the next section we will prove that the measure $m_{k}$ for $k\ge3$
is actually defined by an $L^{2}$-function $M\left(r_{1},r\right)$,
and $dm_{k}^{r_{1}}\left(r\right)=M\left(r_{1},r\right)dr$.

\subsection*{The Brownian Motion}

The Brownian motion is the random walk on $\H$ defined by $B_{t}=\exp\left(-\Delta t\right)$.
The Brownian motion was studied by many authors, and can be analyzed
either by the Helgason-Fourier transform, or by the ``distance to
infinity'' approach used to study the discrete random walk. In any
case, based on \cite{davies1988heat,cammarota2014asymptotic}, we
may write $B_{t}f(x)=\int p(t,r)\left(A_{r}f\right)(x)dr$, with %

\[
p\left(t,r\right)\asymp\frac{t^{-1}r}{\sqrt{1+r+t}}\exp\left(-\frac{\left(r-t\right)^{2}}{4t}\right)\ll t^{-1}r\exp\left(-\frac{\left(r-t\right)^{2}}{4t}\right).
\]
\begin{prop}
\label{prop:Integral and CLT for Brownian}There exist $c>0,t_{0}\ge0$
such that for every $\lambda>0$ and every $t>t_{0}$

\[
\intop_{r:\left|r-t\right|\ge\lambda\sqrt{t}}p(t,r)dr\ll_{t_{0}}e^{-c\lambda^{2}}.
\]
\end{prop}

\begin{proof}
We have
\[
\intop_{r:\left|r-t\right|\ge\lambda\sqrt{t}}p(t,r)dr\le\intop_{-\infty}^{-\lambda}p(t,t+\lambda\sqrt{t})d\lambda'+\intop_{\lambda}^{\infty}p(t,t+\lambda^{\prime}\sqrt{t})d\lambda'
\]
For $r=t-\lambda'\sqrt{t}\le t$, we have
\[
p\left(t,r\right)\ll e^{-\frac{\lambda^{\prime2}}{4}},
\]
so by the standard bound for $\lambda\ge0$
\[
\intop_{-\infty}^{-\lambda}e^{-x^{2}}dx=\intop_{-\infty}^{0}e^{-\left(-\lambda+x\right)^{2}}dx\le e^{-\lambda^{2}}\intop_{-\infty}^{0}e^{-x^{2}}\ll e^{-\lambda^{2}},
\]
we have
\[
\intop_{-\infty}^{-\lambda}p(t,t+\lambda^{\prime}\sqrt{t})d\lambda'\ll e^{-\frac{\lambda^{2}}{4}}.
\]
For $r=t+\lambda'\sqrt{t}\ge t$
\[
p\left(t,r\right)\ll\left(1+\frac{\lambda^{\prime}}{\sqrt{t}}\right)e^{-\frac{\lambda^{\prime2}}{4}},
\]
so for $t\ge t_{0}$
\begin{align*}
\intop_{\lambda}^{\infty}p(t,t+\lambda^{\prime}\sqrt{t})d\lambda' & \ll e^{-\frac{\lambda^{2}}{4}}+\frac{1}{\sqrt{t_{0}}}\intop_{\lambda}^{\infty}\lambda^{\prime}e^{-\frac{\lambda^{\prime2}}{4}}d\lambda'\\
 & \ll_{t_{0}}e^{-\frac{\lambda^{2}}{4}}+\left.\left(e^{-\frac{\lambda^{\prime2}}{4}}\right)\right|_{\lambda}^{\infty}\ll e^{-\frac{\lambda^{2}}{4}}.
\end{align*}
\end{proof}

\section{\label{sec:DeltaToL2}Short Time Bound on the Random Walks}

In this section we show that after a short time both random walks
on $X$ can be described by an $L^{2}$-function, whose norm depend
is bounded is the injectivity radius of $x_{0}$ is bounded away from
$0$. 

It was shown in Section \ref{sec:Preliminaries-and-Notation} that
the $L^{2}-$spectrum of the operator $A_{r}$ constitutes of the
values of $\varphi_{\frac{1}{2}+is}(e^{r}i)$ for $s\in\R$. For the
ease of notation we write $\phi(s,r)=\varphi_{\frac{1}{2}+is}(e^{r}i)$.
\begin{lem}
\label{lem:Technical Lemma0}For any $r$, the following inequality
holds
\[
\left|\phi(s,r)\right|\ll_{r}\left|s\right|^{-1/2}.
\]
\end{lem}

\begin{proof}
Up to a constant, the function $\phi(s,r)$ is equal to $\intop_{0}^{1}\frac{\cos\left(sr\left(1-x\right)\right)}{\sqrt{\cosh r-\cosh r\left(1-x\right)}}dx$
. This function is continuous in $s$, hence we may assume that $\left|s\right|$
is large enough. Write

\[
\intop_{0}^{1}\frac{\cos\left(sr\left(1-x\right)\right)}{\sqrt{\cosh r-\cosh r\left(1-x\right)}}dx=\intop_{0}^{\left|s\right|^{-1}}\frac{\cos\left(sr\left(1-x\right)\right)}{\sqrt{\cosh r-\cosh r\left(1-x\right)}}dx+\intop_{\left|s\right|^{-1}}^{1}\frac{\cos\left(sr\left(1-x\right)\right)}{\sqrt{\cosh r-\cosh r\left(1-x\right)}}dx.
\]
Then since $\lim_{x\to0^{+}}\frac{\sqrt{x}}{\sqrt{\cosh r-\cosh r\left(1-x\right)}}=c_{r}>0$,
for $\left|s\right|$ large enough we have

\[
\left|\intop_{0}^{\left|s\right|^{-1}}\frac{\cos\left(sr\left(1-x\right)\right)}{\sqrt{\cosh r-\cosh r\left(1-x\right)}}dx\right|\le\intop_{0}^{\left|s\right|^{-1}}\frac{1}{\sqrt{\cosh r-\cosh r\left(1-x\right)}}dx\ll_{r}\intop_{0}^{\left|s\right|^{-1}}\frac{1}{\sqrt{x}}dx\ll\frac{1}{\sqrt{\left|s\right|}}.
\]
Analogously, 
\[
\intop_{\left|s\right|^{-1}}^{1}\frac{1}{\left(\cosh r-\cosh r\left(1-x\right)\right)^{3/2}}dx\ll_{r}\sqrt{\left|s\right|}.
\]
Write $G(x)=-\frac{1}{sr}\sin\left(sr\left(1-x\right)\right)$, $F(x)=1/\sqrt{\cosh r-\cosh r\left(1-x\right)}$,
then by integration by parts, 
\begin{align*}
\intop_{\left|s\right|^{-1}}^{1}G'(x)F(x)dx & =G(1)F(1)-G\left(\left|s\right|^{-1}\right)F(\left|s\right|^{-1})-\intop_{\left|s\right|^{-1}}^{1}G(x)F'(x)dx,\\
 & =-G\left(\left|s\right|^{-1}\right)F(\left|s\right|^{-1})-\intop_{\left|s\right|^{-1}}^{1}G(x)F'(x)dx
\end{align*}
and hence,
\begin{align*}
\left|\intop_{\left|s\right|^{-1}}^{1}\frac{\cos\left(sr\left(1-x\right)\right)}{\sqrt{\cosh r-\cosh r\left(1-x\right)}}dx\right| & \ll\left|\frac{1}{ss^{-1}}\right|\left|\frac{1}{\sqrt{\cosh r-\cosh r\left(1-\left|s\right|^{-1}\right)}}\right|+\intop_{\left|s\right|^{-1}}^{1}\frac{1}{\left|s\right|r\left(\cosh r-\cosh r\left(1-x\right)\right)^{3/2}}dx\\
 & \ll_{r}\left|s\right|^{-1/2}+\frac{1}{\left|s\right|}\cdot\sqrt{\left|s\right|}\ll\frac{1}{\sqrt{\left|s\right|}},
\end{align*}
which completes the proof.
\end{proof}
\begin{lem}
\label{lem:Technical Lemma1}For any $x_{0}\in\H$, we have $A_{r_{1}}^{3}\delta_{x_{0}}\in L^{2}\left(\H\right)$.
\end{lem}

\begin{proof}
By Theorem \ref{prop:Plancherel-Helgason}, the Helgason-Fourier transform
of $A_{r_{1}}^{3}$ satisfies $\widehat{A_{r_{1}}^{3}(s)}=\left(\widehat{A_{r_{1}}}(s)\right)^{3}=\phi^{3}\left(s,r_{1}\right)$.
Applying Lemma \ref{lem:Technical Lemma0}, and using the fact that
$\frac{s}{4\pi}\tanh\left(\pi s\right)<s$ implies that
\begin{align*}
\intop_{-\infty}^{\infty}\left|\widehat{A_{r_{1}}^{3}(s)}\right|^{2}\frac{s}{4\pi}\tanh\left(\pi s\right)ds & =\intop_{-\infty}^{\infty}\left|\phi\left(s,r_{1}\right)\right|^{6}\frac{s}{4\pi}\tanh\left(\pi s\right)ds\\
 & \ll_{r}1+\intop_{\left|s\right|>1}\left|s\right|^{-3}\left|s\right|ds<\infty.
\end{align*}
Using the inverse Fourier transform we conclude by Corollary \ref{cor:L2 fourier transform}
that $A_{r_{1}}^{3}=\intop_{r}f(r)A_{r}dr$, with $f(r)$ an $L^{2}$-function.
In particular, $\nn{A_{r_{1}}^{3}\delta_{x_{0}}}=\intop_{r}\left|f(r)\right|^{2}dr<\infty$,
as needed.
\end{proof}
\begin{rem}
For $k=0,1,2,$ the analogous statement is not true. For $k=0,1$,
$A_{r_{1}}^{k}\delta_{x_{0}}$ cannot be considered as a function.
For $k_{0}=2$, $A_{r_{1}}^{2}=\intop_{0}^{r_{1}}g(r)A_{r}dr$, $\intop_{0}^{r}g(r')dr'=\text{acos}((\cosh^{2}(r_{1})-\cosh(r))/\sinh^{2}(r_{1}))/\pi$,
where $g(r)$ is a function on $X$, but not an $L^{2}$ function.
\end{rem}

\begin{lem}
\label{lem:Technical Lemma2}For $k_{0}\ge3$ (respectively for $t_{0}>0$)
there exists a constant $C=C\left(r_{0},r_{1},k_{0}\right)$ (resp.
$C=C\left(r_{0},t_{0}\right)$) such that if $x_{0}\in X$ has a injectivity
radius at least $r_{0}$ then $A_{r_{1}}^{k_{0}}\delta_{x_{0}}\in L^{2}\left(X\right)$
and $\nn{A_{r_{1}}^{k_{0}}\delta_{x_{0}}}\le C$ (resp. $B_{t_{0}}\delta_{x_{0}}\in L^{2}\left(X\right)$
and $\nn{B_{t_{0}}\delta_{x_{0}}}\le C$).
\end{lem}

\begin{proof}
We start with the discrete random walk $A_{r_{1}}^{k_{0}}\delta_{y_{0}}$.
Since $\nn{A_{r_{1}}}\le1$ it is enough to assume that $k_{0}=3$. 

Let $y_{0}\in\H$ be a fixed point covering $x_{0}\in X$. Let $x_{1}\in X$
be a point different from $x_{0}$. We claim that it has a bounded
number $D\ll_{r_{0},k_{0},k_{1}}1$ of points $z_{1},...,z_{D}\in B_{k_{0}r_{0}}\left(y_{0}\right)$
covering $x_{1}$. Since $A_{r_{1}}^{k_{0}}\delta_{y_{0}}\in L^{2}\left(\H\right)$,
it is supported on $B_{k_{0}r_{0}}\left(y_{0}\right)$ and $A_{r_{1}}^{k_{0}}\delta_{x_{0}}$
is the push-forward of $A_{r_{1}}^{k_{0}}\delta_{y_{0}}$ to $X$,
this claim will give the lemma for the discrete random walk. We may
assume that $d_{0}=d\left(x_{0},x_{1}\right)<k_{0}r_{1}$. Let therefore
$z_{1},z_{2},..\in B_{k_{0}r_{1}}\left(y_{0}\right)$ be a sequence
of different points covering $x_{1}$. Then each such point $z_{i}\in\H$
can be associated with another point $y_{i}\in\H$, covering $x_{0}$,
with $d\left(y_{i},z_{i}\right)=d_{0}$. Moreover, we may choose $y_{i}$
such that $y_{i}\ne y_{j}$ for $z_{i}\ne z_{j}$. By the injectivity
radius assumption, $d\left(y_{i},y_{j}\right)\ge2r_{0}$ for $i\ne j$.
All the $y_{i}$'s are contained in the ball $B_{2k_{0}r_{1}}\left(y_{0}\right)$,
and their number is therefore bounded by $\frac{\mu\left(B_{2k_{0}r_{1}}\right)}{\mu\left(B_{r_{0}}\right)}\ll_{r_{0},k_{0},k_{1}}1$.

Now we turn to the Brownian motion. Since $\nn{B_{t}}\le1$ and $B_{t+t'}=B_{t}B_{t'}$
we may assume that $t_{0}$ is small enough so that $p_{2}\left(r,t_{0}\right)$
is decreasing for $r>r_{0}$ and $B_{t_{0}}\delta_{y_{0}}\left(z\right)\le e^{-cd\left(y_{0},z\right)^{2}}$
for some $c=c\left(r_{0},t_{0}\right)>0$ and $d\left(y_{0},z\right)>r_{0}$.

Let $y_{0}\in\H$ be again a fixed point covering $x_{0}\in X$. Let
$x_{1}\in X$ be another point and let $d_{1}=d\left(x_{0},x_{1}\right)$.
Each point $z_{i}$ covering $x_{1}$ satisfies $d\left(y_{0},z_{i}\right)\ge d_{1}$
and the number of points $z_{i}$ covering $x_{1}$ of distance $d\left(y_{0},z_{i}\right)\le r$
is at most $D_{r}\le\frac{\mu\left(B_{r+d_{1}}\right)}{\mu\left(B_{r_{0}}\right)}\ll_{r_{0}}e^{d_{1}+r}$.
Therefore we get the bound: 
\begin{align*}
B_{t_{0}}\delta_{x_{0}}\left(x\right) & =\sum_{z_{i}}B_{t_{0}}\delta_{y_{0}}\left(z_{i}\right)=\sum_{k=0}^{\infty}\sum_{z_{i}:d_{1}+r_{0}k\le d(y_{0},z_{i})\le d_{1}+r_{0}\left(k+1\right)}B_{t_{0}}\delta_{y_{0}}\left(z_{i}\right)\\
 & \le\sum_{k=0}^{\infty}D_{d_{1}+r_{0}\left(k+1\right)}\cdot e^{-c\left(d_{1}+r_{0}k\right)^{2}}\ll_{r_{0}}\sum_{k=0}^{\infty}e^{d_{1}+d_{1}+r_{0}\left(k+1\right)-c\left(d_{1}+r_{0}k\right)^{2}}\\
 & \ll_{r_{0},t_{0}}e^{2d_{1}-c'd_{1}^{2}}.
\end{align*}

For some constant $c'>0$ depending on $r_{0},t_{0}$. Finally, using
the fact that the volume of $x\in X$ with $d\left(x_{0},x\right)\le d_{1}$
is $\ll e^{d_{1}}$,
\[
\Vert B_{t_{0}}\delta_{x_{0}}\Vert_{2}^{2}=\intop_{X}\left|B_{t_{0}}\delta_{x_{0}}(x)\right|^{2}dx\ll_{t_{0},r_{0}}\intop_{d_{1}\ge0}e^{2\left(2d_{1}-c'd_{1}^{2}\right)}\cdot e^{d_{1}}dt\ll_{t_{0},r_{0}}1,
\]
and the lemma is proved.
\end{proof}

\section{\label{sec:Proof-of-Theorem2}Proof of Theorem \ref{thm2}}
\begin{proof}
of Theorem \ref{thm2}. We prove the claim for the discrete random
walk only. The proof for the Brownian motion is analogous, and exploits
Corollary \ref{prop:Integral and CLT for Brownian} instead of Corollary
\ref{cor:integral and CLT for discrete} and the Brownian motion part
of Lemma \ref{lem:Technical Lemma2} instead of its discrete part.

Suppose $rk\alpha<R_{X}-\lambda\sqrt{R_{X}}$. Let $Y=\left\{ y\in X:d\left(x_{0},y\right)>R_{X}-\frac{\lambda}{2}\sqrt{R_{X}}\right\} $.
As $R_{X}\to\infty$, 
\[
1\ge\mu\left(Y\right)/\mu\left(X\right)\ge\left(\mu\left(X\right)-\mu\left(B\left(R_{X}-\frac{\lambda}{2}\sqrt{R_{X}}\right)\right)\right)/\mu\left(X\right)\to1,
\]
so $\mu\left(Y\right)/\mu\left(X\right)\to1$. 

By Corollary \ref{cor:integral and CLT for discrete}, there exists
$c_{1}(r_{1}),C_{1}\left(r_{1}\right)$ such that for $R_{X}$ large
enough,
\begin{align*}
\intop_{Y}\left|A_{r}^{k}\delta_{x_{0}}(x)\right|d\mu & <C_{1}e^{-c_{1}\lambda^{2}}\\
\intop_{X-Y}\left|A_{r}^{k}\delta_{x_{0}}(x)\right|d\mu & >1-C_{1}e^{-c_{1}\lambda^{2}}.
\end{align*}
Therefore,
\begin{align*}
\n{A_{r}^{k}b_{x_{0},r_{0}}-\pi} & =\intop_{Y}\left|A_{r}^{k}\delta_{x_{0}}(x)-\pi(x)\right|d\mu+\intop_{X-Y}\left|A_{r}^{k}\delta_{x_{0}}(x)-\pi(x)\right|d\mu\\
 & \ge\intop_{Y}\left|\pi(x)\right|d\mu-\intop_{Y}\left|A_{r}^{k_{n}}\delta_{x_{0}}(x)\right|d\mu+\intop_{X-Y}\left|A_{r}^{k}\delta_{x_{0}}(x)\right|-\intop_{X-Y}\left|\pi(x)\right|d\mu\\
 & \ge\mu\left(X\right)^{-1}\mu\left(Y\right)-C_{1}e^{-c_{1}\lambda^{2}}+1-C_{1}e^{-c_{1}\lambda^{2}}-\mu\left(X\right)^{-1}\mu\left(X-Y\right)\\
 & =2\mu\left(X\right)^{-1}\mu\left(Y\right)-2C_{1}e^{-c_{1}\lambda^{2}},
\end{align*}
and the first bound follows by letting $R_{X}\to\infty$. Notice that
it does not require the Ramanujan assumption.

For the second bound, recall that we may write $A_{r_{1}}^{k}b_{x_{0},r_{0}}(x)=\intop_{r}\left(A_{r}b_{x_{0},r_{0}}\right)(x)dm_{k}(r)$.
Assume that $kr_{1}\alpha>R_{X}+\lambda\sqrt{R_{X}}$. By Corollary
\ref{cor:integral and CLT for discrete}, for some $c_{2}\left(r_{1}\right)>0$,
for $R_{X}$ large enough (depending on $r_{1},\lambda$),
\[
\intop_{r<R_{X}+\frac{\lambda}{2}\sqrt{R_{X}}}dm_{k}r\ll_{r_{1}}e^{-c_{2}\lambda^{2}}.
\]
As $x_{0}\in X$ has an injectivity radius at least $r_{0}$, by Lemma
\ref{lem:Technical Lemma2}, there exists a constant $C_{3}=C\left(r_{0},r_{1}\right)$
such that $\nn{A_{r_{1}}^{3}\delta_{x_{0}}}\le C_{3}$. 

For every $f\in L_{2}\left(X\right)$, Cauchy\textendash Schwartz
inequality implies that $\n f\le\sqrt{\mu\left(X\right)}\nn f$. Therefore,
\begin{align*}
\n{A_{r_{1}}^{k}f-\pi} & \le\intop_{r}\n{A_{r}\delta_{x_{0}}-\pi}dm_{k}r=\\
 & =\intop_{r<R_{X}+\frac{\lambda}{2}\sqrt{R_{X}}}\n{A_{r}b_{x_{0},r_{0}}-\pi}dm_{k}r+\intop_{r\ge R_{X}+\frac{\lambda}{2}\sqrt{R_{X}}}\n{A_{r}b_{x_{0},r_{0}}-\pi}dm_{k}r\\
 & \le\intop_{r<R_{X}+\frac{\lambda}{2}\sqrt{R_{X}}}2dm_{k}r+\intop_{r\ge R_{X}+\frac{\lambda}{2}\sqrt{R_{X}}}\mu\left(X\right)^{1/2}\nn{A_{r}\left(b_{x_{0},r_{0}}-\pi\right)}dm_{k}r\\
 & \ll_{r_{1}}e^{-c_{2}\lambda^{2}}+\intop_{r\ge R_{X}+\frac{\lambda}{2}\sqrt{R_{X}}}\mu\left(X\right)^{1/2}\left(r+1\right)e^{-r/2}dm_{k}r\\
 & \ll e^{-c_{2}\lambda^{2}}+\mu\left(X\right)^{1/2}\left(R_{X}+\frac{\lambda}{2}\sqrt{R_{X}}+1\right)e^{-\frac{1}{2}\left(R_{X}+\frac{\lambda}{2}\sqrt{R_{X}}\right)}\\
 & \ll e^{-c_{2}\lambda^{2}}+\mu\left(X\right)^{1/2}e^{-\frac{1}{2}R_{X}}\left(R_{X}+\frac{\lambda}{2}\sqrt{R_{X}}+1\right)e^{-\frac{1}{4}\sqrt{R_{X}}}\\
 & \to_{_{R_{X}\to\infty}}e^{-c_{2}\lambda^{2}},
\end{align*}
and the second bound follows.
\end{proof}
\begin{rem}
The theorem holds for $\lambda>0$ such that $R_{X}\gg_{r_{0},r_{1},\lambda}1$.
In other words, $R_{X}$ has to be larger than some constant $R(r_{0},r_{1},\lambda)$
that depends on $r_{0},r_{1}$ and $\lambda$. By fixing $r_{0},r_{1}$,
one can find the relation between this constant and $\lambda$, namely,
it should hold that $\lambda=o\left(\sqrt{R(r_{0},r_{1},\lambda)}\right)$
and $\ln\left(R(r_{0},r_{1},\lambda)\right)=o\left(\lambda\sqrt{R(r_{0},r_{1},\lambda)}\right)$.
\end{rem}

\section{\label{sec:Lp-expanders}$L^{p}$-Bounds}

The above results assume $X$ to be Ramanujan. However, similar results
can be proved in a general setting. In this section and the next one,
we discuss Theorem \ref{thm1} only, but similarly one can elaborate
on Theorem \ref{thm2} as well. 

The following lemma is well known (see \cite{lax1982asymptotic}):
\begin{lem}
\label{lem:Lax-Phillips}The spectrum of the $\Delta$ on $L_{0}^{2}\left(X\right)$
below $1/4$ is discrete, and corresponds to a finite number of eigenvalues
with multiplicities.
\end{lem}

Eigenvalues of $\Delta$ strictly below $1/4$ are called exceptional.
A nice way to measure how far is a representation $V$ of $G$ from
being tempered is to ask what is the minimal $p\geq2$ such that the
$K$-finite matrix coefficients of $G$ on $V$ lie in $L^{p}\left(G\right)$.
The following proposition relates this property to the spectra of
the Laplacian and of the operators $A_{r}$. See \cite{kamber2016lpgraph}
for the corresponding result on graphs.
\begin{prop}
\label{prop:Lp proposition}The following are equivalent for $p\ge2$: 
\begin{enumerate}
\item For every $r\ge0$, the norm of $A_{r}$ on $V$ is bounded by $\left(r+1\right)e^{-r/p}$.
\item Every matrix coefficient of a subrepresentation of $G$ on $L_{0}^{2}\left(\Gamma\backslash G\right)$
with $K=PSO_{2}\left(\R\right)$ fixed vectors is in $L^{p+\epsilon}\left(G\right)$
for every $\epsilon>0$.
\item The spectrum of $\Delta$ on $L_{0}^{2}\left(X\right)$ is bounded
from below by $\frac{1}{4}-\left(\frac{1}{2}-p^{-1}\right)^{2}$.
\end{enumerate}
The equivalence is also true for every $\Delta$-invariant closed
subspace $V\subset L_{0}^{2}\left(X\right)$, where in $(2)$ we look
at the $G-$subrepresentation generated by $V$.
\end{prop}

\begin{proof}
Let $V$ be a $\Delta$-invariant closed subspace of $L_{0}^{2}\left(X\right)$.

The complementary series is determined by a real parameter $t$, $0\le t\le1$,
with corresponding spherical function $\varphi_{t}$. Write $t=\frac{1}{2}+s_{t}^{\prime}$
and $p_{t}=\left(\frac{1}{2}-\left|s_{t}^{\prime}\right|\right)^{-1}$
(with the convention that $0^{-1}=\infty$). The corresponding eigenvalue
of the $\Delta$ on $\varphi_{t}$ is 
\[
\lambda_{t}=t\left(1-t\right)=\frac{1}{4}-s_{t}^{\prime2}=\frac{1}{4}-\left(\frac{1}{2}-p_{t}^{-1}\right)^{2}.
\]
The eigenvalue of $A_{r}$ on $\varphi_{t}$ is:
\[
\varphi_{t}(r)=\frac{1}{\sqrt{2}\pi}r\intop_{-1}^{1}\frac{\exp\left(s_{t}^{\prime}rx\right)}{\sqrt{\cosh r-\cosh rx}}dx.
\]
Recall that for $r>0$, $\cosh r-\cosh\left(rx\right)\ge\left(\cosh r-1\right)\left(1-x^{2}\right)$
holds. Hence
\begin{align*}
\left|\varphi_{t}(r)\right| & =\frac{1}{\pi\sqrt{2}}r\left|\intop_{-1}^{1}\frac{\exp\left(s_{t}^{\prime}rx\right)}{\sqrt{\cosh r-\cosh rx}}dx\right|\le\frac{1}{\pi\sqrt{2}}r\frac{1}{\sqrt{\cosh r-1}}\intop_{-1}^{1}\frac{\exp\left(s_{t}^{\prime}rx\right)}{\sqrt{1-x^{2}}}dx\\
 & \le\frac{1}{\sqrt{2}\pi}r\frac{\exp\left(\left|s_{t}^{\prime}\right|r\right)}{\sqrt{\cosh r-1}}\intop_{-1}^{1}\frac{1}{\sqrt{1-x^{2}}}dx=\frac{1}{\sqrt{2}}r\left(\cosh r-1\right)^{-1/2}\exp\left(\left|s_{t}^{\prime}\right|r\right)\\
 & \le\left(r+1\right)e^{-r\left(\frac{1}{2}-\left|s_{t}^{\prime}\right|\right)}=\left(r+1\right)e^{-r/p_{t}}.
\end{align*}
This proves an implication from $(3)$ to $(1)$.

We also want to give a lower bound on $\left|\varphi_{t}(r)\right|$.
Let $f(x)=\cosh r-\cosh\left(r(1-x)\right)$. For $x\ge0$, by the
Taylor series $f(x)=xr\sinh r-\frac{1}{2}r^{2}x^{2}\cosh\left(r\left(1-x'\right)\right)$,
$0\le x'\le x$, so $f(x)\le xr\sinh r$. So for a fixed $\epsilon>0$,
and $r\ge1$,
\begin{align}
\varphi_{t}(r) & =\frac{1}{\sqrt{2}\pi}r\intop_{-1}^{1}\frac{\exp\left(s_{t}^{\prime}rx\right)}{\sqrt{\cosh r-\cosh rx}}dx\gg r\intop_{0}^{2}\frac{\exp\left(s_{t}^{\prime}r\left(1-x\right)\right)}{\sqrt{\cosh r-\cosh r\left(1-x\right)}}dx\nonumber \\
 & \ge r\intop_{0}^{\epsilon}\frac{\exp\left(s_{t}^{\prime}r\left(1-x\right)\right)}{\sqrt{\cosh r-\cosh r\left(1-x\right)}}dx\nonumber \\
 & \ge\frac{\sqrt{r}e^{s_{t}^{\prime}r}}{\sqrt{\sinh r}}\intop_{0}^{\epsilon}\frac{\exp\left(-s_{t}^{\prime}rx\right)}{\sqrt{x}}dx\gg\frac{e^{s_{t}^{\prime}r\left(1-\epsilon\right)}}{\sqrt{\sinh r}}\sqrt{\epsilon}\gg\sqrt{\epsilon}e^{-r\left(\frac{1}{2}-\left|s_{t}^{\prime}\right|(1-\epsilon)\right)}.\label{eq:Lp pf2}
\end{align}
This implies that if for every $r\ge0$, $\varphi_{t}(r)\le\left(r+1\right)e^{-r\left(\frac{1}{2}-S\right)}$
then $\left|s_{t}^{\prime}\right|\le S$. This proves the implication
from $(1)$ to $(3)$. 

Arguing as in Proposition \ref{prop:ArBoundsImpliesRamanujan}, we
see that $(2)$ is equivalent to:
\begin{itemize}
\item For every $f,f'\in V$ and for every $\epsilon>0$,
\begin{equation}
\int_{r\ge0}e^{r}\left|\left\langle f,A_{r}f'\right\rangle \right|^{p+\epsilon}dr<\infty\label{eq:Lp pf3}
\end{equation}
\end{itemize}
We can immediately see that as in Proposition \ref{prop:ArBoundsImpliesRamanujan},
this proves the implication from $(1)$ to $(2)$. 

Assume now that $(1)$ and $(3)$ do not hold for $p=p_{0}\ge2$.
By Lemma \ref{lem:Lax-Phillips}, there is an eigenvector $f\in V$
which satisfies $\left\langle f,A_{r}f\right\rangle =\varphi_{t}\left(r\right)$,
for some $\varphi_{t}$, with $p_{t}>p_{0}$. By Equation \ref{eq:Lp pf2},
for some $\delta>0$ and for $r$ large enough $\left|\left\langle f,A_{r}f\right\rangle \right|\gg_{\delta}e^{-r\left(1/p_{0}+\delta\right)}$.
Then Equation \ref{eq:Lp pf3} does not hold, and $(2)$ does not
hold.
\end{proof}
By Lemma \ref{lem:Lax-Phillips}, for each $X$ there is a minimal
$p_{0}$ satisfying the equivalent conditions of Proposition \ref{prop:Lp proposition}.
Denote it by $p_{0}\left(X\right)$. For example, Selberg's lower
bound $3/16$ implies that for each $X$ corresponding to a congruence
subgroup of $SL_{2}\left(\Z\right)$, $p_{0}\left(X\right)\le4$.
Further improvements (see e.g. \cite{sarnak1995selberg}) improve
this bound as well. Without any further information about $X$, we
can say the following:
\begin{thm}
\label{thm:Lp-expanders}Let $r_{0}>0$ be fixed. Let $p=p_{0}\left(X\right)$
and assume $R_{X}\ge1$. Let $x_{0}\in X$ be a point with injectivity
radius at least $r_{0}$. Then for every $\gamma>0$
\[
\mu_{X}\left(x\in X:d_{X}\left(x,x_{0}\right)\ge\frac{p}{2}\left(R_{X}+\gamma\ln\left(R_{X}\right)\right)\right)/\mu\left(X\right)\ll_{p,r_{0}}\left(1+\gamma^{2}\right)R_{X}^{2-\gamma}.
\]
\end{thm}

\begin{proof}
The proof is essentially the same as the proof of Theorem \ref{thm1},
and we only write the differences. Instead of choosing $r'=R_{X}+\gamma\ln\left(R_{X}\right)-r_{0}$,
choose $r'=\frac{p}{2}\left(R_{X}+\gamma\ln\left(R_{X}\right)\right)-r_{0}$.
Then:
\begin{align*}
\nn{A_{r^{\prime}}b_{x_{0},r_{0}}-\pi} & =\nn{A_{r^{\prime}}\left(b_{x_{0},r_{0}}-\pi\right)}\le\left(r'+1\right)e^{-r^{\prime}/p}\nn{b_{x_{0},r_{0}}-\pi}\\
 & \ll_{r_{0}}\left(r'+1\right)e^{-\frac{1}{2}R_{X}-\frac{1}{2}\gamma\ln\left(R_{X}\right)+r_{0}/p}\\
 & \ll_{r_{0},p}\frac{\frac{p}{2}\left(R_{X}+\gamma\ln\left(R_{X}\right)\right)-r_{0}}{R_{X}}\mu\left(X\right)^{-1/2}e^{-\frac{1}{2}(\gamma-2)\ln\left(R_{X}\right)}\\
 & \ll_{r_{0},p}\left(1+\gamma\right)\mu\left(X\right)^{-1/2}e^{-\frac{1}{2}(\gamma-2)\ln\left(R_{X}\right)}.
\end{align*}
The rest of the proof is the same.
\end{proof}

\section{\label{sec:Coverings}Covers}

Let $X_{0}=\Gamma_{0}\backslash\H$. Then a finite index subgroup
$\Gamma_{X}\subset\Gamma_{0}$ defines a cover $X=\Gamma_{X}\backslash\H$
of $X_{0}$, with cover map $\rho:X\to X_{0}$. The pull-back $\rho^{*}:L^{2}\left(X_{0}\right)\to L^{2}\left(X\right)$
defines a closed subspace $\rho^{*}L^{2}\left(X_{0}\right)\subset L^{2}\left(X\right)$.
Denote the orthogonal complement of $\rho^{*}L^{2}\left(X_{0}\right)$
in $L^{2}\left(X\right)$ by $L^{2}\left(X/X_{0}\right).$

For $p>2$, denote by $m\left(X,p\right)$ the dimension of the space
spanned by eigenvectors of $L^{2}\left(X/X_{0}\right)$ whose matrix
coefficients are not in $L^{p'}$ for every $p'<p$ but are in $L^{p'}$
for $p'>p$. Denote also $M(X,p)=\sum_{p'\ge p}m\left(X,p'\right)$.

A cover $\rho:X\to X_{0}$ is called normal if $\Gamma_{X}\subset\Gamma_{X_{0}}$
is a normal subgroup. Equivalently, a cover $\rho:X\to X_{0}$ is
normal if there exists a group $H$ acting on $X$ such that $\rho\left(x\right)=\rho\left(y\right)$
if and only if $x$ and $y$ are on the same $H$-orbit. We call $H$
the cover group.

Our main result about covers is as the following theorem. Note that
if $X$ is an $N$-cover of $X_{0}$ then $\mu\left(X\right)=N\cdot\mu\left(X_{0}\right)$.
Therefore, $\mu\left(X\right)\asymp_{X_{0}}N$ and $R_{X}=\ln\left(N\right)+O_{X_{0}}\left(1\right)$. 
\begin{thm}
\label{thm:Normal Coverings Theorem}Let $r_{0}>0$ be fixed, and
let $X_{0}$ be a fixed quotient. Let $\rho_{q}:X_{q}\to X_{0}$ be
family of normal $N_{q}$-covers, with $N_{q}\to\infty$ as $q\to\infty$.

Assume that $g:\mathbb{R}_{+}\to\mathbb{R}_{+}$ is non-decreasing
function satisfying:
\begin{enumerate}
\item For some fixed $\delta>2$ and for $R$ large enough, $g(R)\ge R+\delta\ln R$.
\item Either
\begin{equation}
g^{3}\left(\ln\left(N_{q}\right)\right)\sum_{p:m\left(X_{q},p\right)\ne0}e^{-2g\left(\ln\left(N_{q}\right)\right)/p_{i}}m\left(X_{q},p_{i}\right)=o(1),\label{eq:Density requirement0}
\end{equation}
or
\begin{align}
g^{3}\left(\ln\left(N_{q}\right)\right)\intop_{2}^{\infty}M\left(X_{q},p\right)e^{-2g\left(\ln\left(N_{q}\right)\right)/p}p^{-2}dp & =o\left(1\right),\,\text{and}\label{eq:Density requirement}\\
g^{2}\left(\ln\left(N_{q}\right)\right)\lim_{p\to2,p>2}M\left(X_{q},p\right)e^{-g\left(\ln\left(N_{q}\right)\right)} & =o(1)\label{eq:Density requirement2}
\end{align}
\end{enumerate}
For every $q$, let $x_{0}^{(q)}\in X_{q}$ be a point such that its
projection $\rho_{q}(x_{0}^{(q)})$ to $X_{0}$ has injectivity radius
at least $r_{0}$. Then
\[
\mu\left(x\in X_{q}:d_{X_{q}}\left(x,x_{0}^{(q)}\right)\ge g\left(\ln\left(N_{q}\right)\right)\right)/\mu\left(X_{q}\right)=o\left(1\right),
\]
where the implied constant depends on $X_{0},\left\{ X_{q}\right\} ,r_{0}$
and $g$.
\end{thm}

Before proving the theorem, let us study its corollaries.
\begin{defn}
\label{def:Density-Definition}We say that a family of covers $\left\{ X_{q}\right\} $
of $X_{0}$ satisfies a \emph{density condition with parameter $A$}
if for every $\epsilon>0$, for each $p>2$, 
\[
M\left(X,p\right)\ll_{\epsilon,\{X_{q}\},X_{0}}CN^{1-A\left(p-2\right)/p+\epsilon},
\]
and furthermore
\begin{itemize}
\item The number of exceptional eigenvalues $\lim_{p\to2,p>2}M\left(X,p\right)=\sum_{p>2}m\left(X,p\right)$
of $X_{q}$ is $\ll_{\{X_{q}\},X_{0}}N$.
\item There exists $p_{\text{max}}$ such that $M\left(X,p_{\text{max}}\right)=0$.
\end{itemize}
\end{defn}

The assumption that the number of exceptional eigenvalues is $O(N)$
is well known to hold in the arithmetic case (see \cite{sarnak1990diophantine}).
There are two main instances of such density results:
\begin{enumerate}
\item The case $A=1$: in this case we may simply write $M\left(X,p\right)\ll_{\epsilon,X_{0}}N^{2/p+\epsilon}$.
This is known to hold for a wide range of cases, including the congruence
subgroups of $SL_{2}\left(\mathbb{Z}\right)$ and all cocompact arithmetic
lattices in $SL_{2}\left(\mathbb{R}\right)$ (See \cite{sarnak1990diophantine,sarnak1991bounds}
for the uniform case and \cite{huntley1993density} for $SL_{2}\left(\mathbb{Z}\right)$).
The corresponding result for LPS graphs are implicitly contained in
\cite[Section 4.4]{davidoff2003elementary}. In this case, for prime
congruence, one may find $p_{\text{max}}$ by using lower bounds on
the dimensions of representations of $SL_{2}\left(F_{q}\right)$ (See
\cite{sarnak1991bounds}). 
\item The case $A>1$: this case requires deeper results in analytic number
theory, and applies to congruence subgroups of $SL_{2}\left(\mathbb{Z}\right)$.
In this case, $p_{\text{max}}$ is essentially bounded by $\frac{2}{1-A^{-1}}$,
and there is no need to restrict to prime congruence. See \cite{iwaniec1985density},
and \cite{humphries2016density} and the references therein for recent
results .
\end{enumerate}
\begin{cor}
\label{cor:Density Theorem Corollary}Let $\rho:X_{q}\to X_{0}$ be
family of normal $N_{q}$-covers, with $N_{q}\to\infty$. Assume the
family satisfies a density condition with parameter $A\ge1$. 

Let $x_{0}^{(q)}\in X_{q}$ be a point such that its projection $\rho_{q}(x_{0}^{(q)})$
to $X_{0}$ has injectivity radius at least $r_{0}$. Then for every
$\epsilon_{0}>0$
\[
\mu\left(x\in X_{q}:d_{X_{q}}\left(x,x_{0}^{(q)}\right)\ge R_{X_{q}}\left(1+\epsilon_{0}\right)\right)/\mu\left(X_{q}\right)=o\left(1\right).
\]
\end{cor}

\begin{proof}
One should verify Inequalities \ref{eq:Density requirement},\ref{eq:Density requirement2}.
We may assume $A=1$.

Let $g(R)=\left(1+\epsilon_{0}\right)R$, then for $\epsilon>0$ small
enough with respect to $\epsilon_{0}$ it holds that 
\begin{align*}
 & g^{3}\left(\ln\left(N_{q}\right)\right)\intop_{2}^{\infty}M\left(X,p\right)e^{-2g\left(\ln\left(N_{q}\right)\right)/p}p^{-2}dp\\
 & \ll_{\epsilon_{0}.\left\{ X_{q}\right\} ,X_{0}}\left(1+e_{0}\right)^{3}\ln^{3}\left(N_{q}\right)\intop_{2}^{p_{\text{max}}}N_{q}^{2/p+\epsilon}e^{-2(1+\epsilon_{0})\ln\left(N_{q}\right)/p}p^{-2}dp\\
 & \ll_{\epsilon_{0}.\left\{ X_{q}\right\} ,X_{0}}\ln^{3}\left(N_{q}\right)\intop_{2}^{p_{\text{max}}}N_{q}^{2/p+\epsilon}N_{q}^{-2(1+\epsilon_{0})/p}p^{-2}dp\\
 & \ll_{X_{0}}\ln^{3}\left(N_{q}\right)\intop_{2}^{p_{\text{max}}}N_{q}^{\epsilon-2\epsilon_{0}/p_{\text{max}}}p^{-2}dp\\
 & \ll\ln^{3}\left(N_{q}\right)N_{q}^{\epsilon-2/p_{\text{max}}\epsilon_{0}}\to_{N_{q}\to\infty}0.
\end{align*}
In addition,
\begin{align*}
 & g^{2}\left(\ln\left(N_{q}\right)\right)\lim_{p\to2,p>2}M\left(X_{q},p\right)e^{-g\left(\ln\left(N_{q}\right)\right)}\\
 & \ll_{\left\{ X_{q}\right\} ,X_{0}}\left(1+\epsilon_{0}\right)^{2}\ln^{2}\left(N_{q}\right)N_{q}^{1-\left(1+\epsilon_{0}\right)}\\
 & \ll_{\epsilon_{0}}\ln^{2}\left(N_{q}\right)N_{q}^{\epsilon_{0}}\to_{N_{q}\to\infty}0.
\end{align*}
\end{proof}
\begin{rem}
\label{rem:covering remark}The density theorems with parameter $A>1$
are not far from proving that there exists $C>0$ such that $\mu\left(x\in X_{q}:d_{X_{q}}\left(x,x_{0}^{(q)}\right)\ge R_{X_{q}}+C\ln R_{X_{q}}\right)/\mu\left(X_{q}\right)=o\left(1\right)$.
The required bound is that for some $\epsilon_{2}>0$,$C_{2}>0$ and
every $2<p<p+\epsilon_{0}$, M$\left(X_{q},p\right)\ll\ln^{C_{2}}\left(N\right)\cdot N^{2/p}$.
\end{rem}

Let us turn to the proof of Theorem \ref{thm:Normal Coverings Theorem}.
It will depend on the following two Lemmas.
\begin{lem}
\label{lem:Covering Lemma 1} Let $\rho:X\to X_{0}$ be an $N$-cover,
$U=\rho^{*}L_{0}^{2}\left(X'\right)\subset L_{0}^{2}\left(X\right)$
be the space of functions pulled back from $X_{0}$ to $X$ and let
$P_{U}$ be the orthogonal projection onto $U$. Let $x_{0}\in X$
be a point such that its projection to $X_{0}$ has injectivity radius
at least $r_{0}$. Then 
\[
\nn{P_{U}\left(b_{x_{0},r_{0}}\right)}=N^{-1/2}\nn{b_{x_{0},r_{0}}}.
\]
\end{lem}

\begin{proof}
We have 
\[
\nn{P_{U}\left(b_{x_{0},r_{0}}\right)}=\max_{u\in U,\nn u=1}\left\langle u,b_{x_{0},r_{0}}\right\rangle =\max_{u'\in L^{2}\left(x'\right),\nn{\rho^{*}u'}=1}\left\langle \rho^{*}u',b_{x_{0},r_{0}}\right\rangle .
\]
 But $\nn{\rho^{*}u'}^{2}=N\nn{u'}^{2}$ and $\left\langle \rho^{*}u',b_{x_{0},r_{0}}\right\rangle =\left\langle u',b_{\rho\left(x_{0}\right),r_{0}}\right\rangle $.
So
\[
\nn{P_{U}\left(b_{x_{0},r_{0}}\right)}=\max_{u'\in L^{2}\left(x'\right),\nn{u'}=N^{-1/2}}\left\langle u',b_{\rho\left(x_{0}\right),r_{0}}\right\rangle =N^{-1/2}\nn{b_{\rho\left(x_{0}\right),r_{0}}}=N^{-1/2}\nn{b_{x_{0},r_{0}}}.
\]
\end{proof}
\begin{lem}
\label{lem:Covering Lemma 2}Let $\rho:X\to X_{0}$ be a normal $N$-cover,
with cover group $H$. Let $W\subset L^{2}\left(X\right)$ be a finite
dimensional $H$-invariant subspace and $P_{W}$ the orthogonal projection
onto this subspace. Let $x_{0}\in X$ be a point such that its projection
to $X_{0}$ has injectivity radius at least $r_{0}$. Then 
\[
\nn{P_{W}\left(b_{x_{0},r_{0}}\right)}\le\sqrt{\frac{\dim W}{N}}\nn{b_{x_{0},r_{0}}}.
\]
\end{lem}

\begin{proof}
Let $u_{1},...,u_{\dim W}$ be an orthonormal basis of $W$. Then
\[
\left\Vert P_{W}\left(b_{x_{0},r_{0}}\right)\right\Vert _{2}^{2}=\sum_{i=1}^{\dim W}\left|\left\langle u_{i},b_{x_{0},r_{0}}\right\rangle \right|^{2}.
\]

On the other hand, the points $hx_{0}$, where $h\in H$, are all
distinct, the balls $B_{r}\left(hx_{0}\right)$ of radius $r_{0}$
around them are disjoint, and since $W$ is $H$-invariant for each
$h\in H$ 
\[
\left\Vert P_{W}\left(b_{hx_{0},r_{0}}\right)\right\Vert _{2}^{2}=\left\Vert P_{W}\left(b_{x_{0},r_{0}}\right)\right\Vert _{2}^{2}.
\]
so
\begin{align*}
N\left\Vert P_{W}\left(b_{x_{0},r_{0}}\right)\right\Vert _{2}^{2} & =\sum_{h\in H}\sum_{i=1}^{\dim W}\left|\left\langle u_{i},b_{hx_{0},r_{0}}\right\rangle \right|^{2}\\
 & \le\sum_{i=1}^{\dim W}\sum_{h\in H}\left\Vert u_{i}|_{B_{r}\left(hx_{0}\right)}\right\Vert _{2}^{2}\left\Vert b_{x_{0},r_{0}}\right\Vert _{2}^{2}\\
 & =\left\Vert b_{x_{0},r_{0}}\right\Vert _{2}^{2}\sum_{i=1}^{\dim W}\sum_{h\in H}\left\Vert u_{i}|_{B_{r}\left(hx_{0}\right)}\right\Vert _{2}^{2}\\
 & \le\left\Vert b_{x_{0},r_{0}}\right\Vert _{2}^{2}\sum_{i=1}^{\dim W}\left\Vert u_{i}\right\Vert _{2}^{2}=\dim W\left\Vert b_{x_{0},r_{0}}\right\Vert _{2}^{2}.
\end{align*}
\end{proof}
\begin{proof}
of Theorem \ref{thm:Normal Coverings Theorem}. To avoid cumbersome
notations we do not use the index $q$ in the proof below.

By the proof of Theorem \ref{thm1} one should prove the following
inequality for $r=g\left(R_{X}\right)$, 
\[
\nn{A_{r}\left(b_{x_{0},r_{0}}-\pi\right)}^{2}=o\left(N^{-1}\right).
\]

Let $\left\{ p_{i}\right\} _{i=1}^{T}$ be the set of $p$-values
(without multiplicities) of exceptional eigenvalues of $L^{2}\left(X/X_{0}\right)$,
i.e., the $p$ such that the corresponding matrix coefficient is not
in $L^{p'}$ for every $p'<p$ but are in $L^{p'}$ for every $p'>p$.
Let $V_{i}$ be the vector space of eigenvectors with $p$-value $p_{i}$.
Let $p_{0}=2$ and $V_{0}$ the orthogonal complement of the $V_{i}$
in $L^{2}\left(X/X_{0}\right)$. Then for $i=0,...,T$, the norm of
$A_{r}$ on $V_{i}$ is bounded by $\left(r+1\right)e^{-r/p_{i}}$.

We have the decomposition
\[
L^{2}\left(X\right)=\text{span}\left\{ \pi\right\} \oplus\rho^{*}L_{0}^{2}\left(X_{0}\right)\oplus V_{0}\oplus V_{1}\oplus...\oplus V_{T}.
\]

Decompose $b_{x_{0},r_{0}}=\pi+u+v_{0}+...+v_{T}$. For $i=1,...,T$,
denote $m\left(X,p_{i}\right)=\dim sV_{i}$. We have

\begin{align*}
\nn u^{2} & =N^{-1}\nn{b_{x_{0},r_{0}}}\ll_{r_{0}}N^{-1}\\
\nn{v_{0}}^{2} & \le\nn{b_{x_{0},r_{0}}}^{2}\ll_{r_{0}}1\\
\nn{v_{i}}^{2} & \le N^{-1}m\left(X,p_{i}\right)\nn{b_{x_{0},r_{0}}}\ll_{r_{0}}N^{-1}m\left(X.p_{i}\right).
\end{align*}
The first equality follow from Lemma \ref{lem:Covering Lemma 1},
the second inequality is straightforward, and the third inequality
follows from Lemma \ref{lem:Covering Lemma 2}. 

Then for $r=g\left(R_{X_{q}}\right)$,
\begin{equation}
\nn{A_{r}\left(b_{x_{0},r_{0}}-\pi\right)}^{2}=\nn{A_{r}u}^{2}+\nn{A_{r}v_{0}}^{2}+\sum_{i=1}^{T}\nn{A_{r}v_{0}}^{2}.\label{eq:normal covering proof}
\end{equation}

Therefore one should prove that the RHS of Equation \ref{eq:normal covering proof}
is $O\left(N^{-1}\right)$.

Since $\nn u^{2}\ll_{r_{0}}N^{-1}$ and $X_{0}$ has some $p_{0}\left(X_{0}\right)$
and the first summand of Equation \ref{eq:normal covering proof}
is $o\left(N^{-1}\right)$. 

Since $\nn{v_{0}}^{2}\ll_{r_{0}}1$ and for some $\delta>2$, and
$R$ large enough $g(R)\ge R+\delta\ln R$, the second summand Equation
\ref{eq:normal covering proof} is $o\left(N^{-1}\right)$.

For the third summand, we have 
\[
\sum_{i=1}^{T}\nn{A_{r}v_{0}}^{2}\le N^{-1}\left(r+1\right)^{2}\sum_{i=1}^{T}e^{-2r/p_{i}}m\left(X,p_{i}\right)m.
\]
This proves that if Inequality \ref{eq:Density requirement0} holds
then the third summand of Equation \ref{eq:normal covering proof}
is $o(N^{-1})$.

Notice that for $1\le i\le T$, $m\left(X,p_{i}\right)=M\left(X,p_{i}\right)-M\left(X,p_{i+1}\right)$,
with $M\left(X,p_{T+1}\right)=0$. Then 
\begin{align*}
\sum_{i=1}^{T}\nn{A_{r}v_{0}}^{2} & \le N^{-1}\left(r+1\right)^{2}\sum_{i=1}^{T}e^{-2r/p_{i}}m\left(X,p_{i}\right)\\
 & =N^{-1}\left(r+1\right)^{2}\sum_{i=1}^{T}e^{-2r/p_{i}}\left(M\left(X,p_{i}\right)-M\left(X,p_{i+1}\right)\right)\\
 & =N^{-1}\left(r+1\right)^{2}\left(M\left(X,p_{1}\right)e^{-2r/p_{1}}+\sum_{i=2}^{T}M\left(X,p_{i}\right)\left(e^{-2r/p_{i}}-e^{-2r/p_{i-1}}\right)\right)\\
 & \le N^{-1}\left(r+1\right)^{2}\left(M\left(X,p_{1}\right)e^{-2r/p_{1}}+\sum_{i=1}^{T}M\left(X,p_{i}\right)2r\left(p_{i}-p_{i-1}\right)e^{-2r/p_{i}}p_{i-1}^{-2}\right),
\end{align*}
Where we used $\left(e^{-2r/p_{i}}-e^{-2r/p_{i-1}}\right)=2r\left(p_{i}-p_{i-1}\right)e^{-2r/p'}p^{\prime-2}$,
for some $p_{i-1}\le p^{\prime}\le p_{i}$. 

By adding arbitrary $p_{i}$-s with $m\left(X,p_{i}\right)=0$ we
may conclude 
\begin{align*}
\sum_{i=1}^{T}\nn{A_{r}v_{0}}^{2} & \le N^{-1}\left(r+1\right)^{2}\left(\lim_{p_{i}\to2,p_{i}>2}M\left(X,p_{i}\right)e^{-r}+2r\intop_{2}^{\infty}M\left(X,p\right)e^{-2r/p}p^{-2}dp\right).
\end{align*}
This proves that if Inequalities \ref{eq:Density requirement} and
\ref{eq:Density requirement2} hold then the third summand in Equation
\ref{eq:normal covering proof} is $o\left(N^{-1}\right)$.
\end{proof}

\section{\label{sec:Concentration of distance}Appendix I: Isoperimetric Inequalities
and Concentration of Distance from a Fixed Vertex}

The bounds we have allows us to prove the following isoperimetric
inequality. Similar bounds are well known (see \cite[Theorem 4.1]{gromov1983topological}).
\begin{lem}
Let $X=\Gamma\backslash\H$ be a quotient, and $p=p_{0}\left(X\right)$
as defined in Proposition \ref{prop:Lp proposition}. For $r\ge0$,
denote $\kappa_{r,p}=\left(r+1\right)^{2}e^{-2r/p}$. For a closed
set $Y\subset X$, let 
\begin{align*}
Y_{r} & =\left\{ x\in X\mid d\left(x,Y\right)\le r\right\} ,
\end{align*}
and denote $c=\mu\left(Y\right)/\mu\left(X\right)$ and $c'=\mu\left(Y_{r}\right)/\mu\left(X\right)$.
Then
\[
c'\ge\frac{c}{\left(\kappa_{r,p}(1-c)+c\right)}\text{, and hence also }c\le\frac{\kappa_{r,p}c'}{1-c'+\kappa_{r,p}c'}.
\]
\end{lem}

\begin{rem}
For $ck_{r,p}^{-1}$ small $c'\gg\frac{e^{2r/p}}{\left(r+1\right)^{2}}c$.
So for $p=2$, up to an $\left(r+1\right)^{-2}$ factor, the growth
of small sets is the best possible, i.e. the size of the radius $r$-ball. 
\end{rem}

\begin{rem}
The result of \cite[Theorem 4.1]{gromov1983topological}, which is
more general and works for all surfaces, not necessarily hyperbolic,
essentially replaces the exponent $2/p=1-\sqrt{1-4\lambda}$ by $\sqrt{\lambda}$,
so the results above are asymptotically better for the relevant domain
$0\le\lambda\le1/4$.
\end{rem}

\begin{proof}
We may assume $\mu\left(Y\right)>0$. Let $b_{Y}\in L^{1}\left(Y\right)$
be defined by 
\[
b_{Y}=\begin{cases}
\mu^{-1}\left(x\right) & x\in Y\\
0 & x\notin Y
\end{cases}.
\]
Then $\left\Vert b_{Y}\right\Vert _{1}=1$, $\left\Vert b_{Y}\right\Vert _{2}^{2}=\mu^{-1}\left(Y\right)$,
$\n{A_{r_{0}}b_{Y}}=1$ and $\text{supp}\left(A_{r}b_{Y}\right)\subset Y_{r}$,
so $\nn{A_{r}b_{Y}}^{-2}\ge\frac{1}{\mu\left(Y_{r}\right)}$, i.e.
\[
\mu\left(Y_{r}\right)\ge\nn{A_{r}b_{Y}}^{2}.
\]
Decompose $b_{Y}=\pi+b$, with 
\[
\nn b^{2}=\nn{b_{Y}}^{2}-\nn{\pi}^{2}=\frac{1}{\mu\left(Y\right)}-\frac{1}{\mu\left(X\right)}=\frac{1-c}{\mu\left(Y\right)}.
\]
We have 

\begin{align*}
\nn{A_{r_{0}}b_{Y}}^{2} & =\nn{A_{r}b}^{2}+\nn{A_{r}\pi}^{2}\\
 & \le\left(r+1\right)^{2}e^{-2r/p}\nn b^{2}+\nn{\pi}^{2}\\
 & \le\left(r+1\right)^{2}e^{-2r/p}(1-c)\mu^{-1}\left(Y\right)+\mu^{-1}\left(X\right)\\
 & =\left(\kappa_{r,p}(1-c)+c\right)\mu^{-1}\left(Y\right).
\end{align*}
Combining the two inequalities we get 
\[
\mu\left(Y_{r}\right)\ge\nn{A_{r_{0}}b_{Y}}^{-2}\ge\frac{c}{\left(\kappa_{r,p}(1-c)+c\right)}\mu\left(X\right).
\]
The other inequality in the theorem follows from the first one.
\end{proof}
We may now state the following concentration of distance theorem:
\begin{thm}
\label{thm:concentration of Distance}There exists $a=a\left(p_{0}\left(X\right)\right)>0$
such that for each $x_{0}\in X$ there exists $R_{X,x_{0}}$ such
that for every $\gamma>0$:
\[
\mu\left(x\in X\mid\left|d_{X}\left(x,x_{0}\right)-R_{X,x_{0}}\right|\ge\gamma\right)/\mu\left(X\right)\ll_{p_{0}\left(X\right)}a^{-\gamma}.
\]

By Theorem \ref{thm1} if $X$ is Ramanujan and $x_{0}$ has injectivity
radius $r_{0}$ then $R_{X,x_{0}}$ satisfies $R_{X}\le R_{X,x_{0}}\le R_{X}+\left(2+\epsilon\right)\ln R_{X}$).
\end{thm}

\begin{proof}
For $r\ge0$ denote 
\[
Y(r)=\left\{ x\in X\mid d\left(x,x_{0}\right)\le r\right\} .
\]
Choose $R_{X,x_{0}}$ to be such that 
\[
\mu\left(Y\left(R_{X,x_{0}}\right)\right)=\frac{1}{2}\mu\left(X\right).
\]
Let $Y=Y(R_{X,x_{0}}-\gamma)$. Then $Y_{\gamma}=Y\left(R_{X,x_{0}}\right)$
and

\begin{align*}
\mu\left(Y\right)\mu\left(X\right)^{-1} & \le\frac{k_{\gamma,p}\frac{1}{2}}{1-\frac{1}{2}+k_{\gamma,p}\frac{1}{2}}=\frac{k_{\gamma,p}}{1+k_{\gamma,p}}\le k_{\gamma,p}.
\end{align*}
Let $Z=Y(R_{X,x_{0}}+\gamma)$. Then $Y\left(R_{X,x_{0}}\right)_{\gamma}=Z$
and
\begin{align*}
\mu\left(Z\right)\mu\left(X\right)^{-1} & \ge\frac{\frac{1}{2}}{\left(\kappa_{r,p}(1-\frac{1}{2})+\frac{1}{2}\right)}=\frac{1}{1+k_{\gamma,p}}.
\end{align*}
Hence
\[
1-\mu\left(Z\right)\mu\left(X\right)^{-1}\le\frac{k_{\gamma,p}}{1+k_{\gamma,p}}\le k_{\gamma,p}
\]
And finally,
\begin{align*}
\mu\left(x\in X:\left|d_{X}\left(x,x_{0}\right)-R_{X,x_{0}}\right|\le\gamma\right)/\mu\left(X\right) & =1-\mu\left(Z\right)\mu\left(X\right)^{-1}-\mu\left(Y\right)\mu\left(X\right)^{-1}\\
 & \le2k_{\gamma,p}.
\end{align*}
We finish by noting that there exists $a=a\left(p\right)$ such that
\[
k_{\gamma,p}\ll_{p}a^{-\gamma}.
\]
\end{proof}

\section{\label{sec:Cutoof and Comparison-with-the-flat}Appendix II: Comparison
with the Flat case}

In \cite[Section 3C]{diaconis1988group}, Diaconis analyses the random
walk on the Cayley graph of $\Z/N\Z$ with respect to the generators
$\pm1$, and shows that it does not have a cutoff. Namely, he shows
that the time $t_{0}^{T}$ until the random walk satisfies $\n{p_{N}^{T}-\pi}\le e^{-T}$
is $\Theta\left(N^{2}T\right)$.

We will similarly analyze the Brownian random walk on the torus $a\Z\backslash\R$
where $a>0$, and show it does not have a cutoff as $a\to\infty$.
Namely, we will show that time until the time $t_{0}^{T}$ until the
random walk satisfies $\n{p_{a}^{T}-\pi}\le e^{-T}$ is $\Theta\left(a^{2}T\right)$.
Similar analysis shows that the Brownian random walk on quotients
of $\R^{n}$ by $a\mathbb{Z}^{n}$ does not express a cutoff as $a\to\infty$. 

It is also worth mentioning that the ``distance $r_{1}$'' discrete
random walk on $a\Z\backslash\R$ does not even converge in $L^{1}$
to the uniform probability, since it remains discrete. For higher
dimensions the ``distance $r_{1}$'' random walk does converge to
the uniform probability (for reasons similar to Section \ref{sec:DeltaToL2}),
but does not express a cutoff by the central limit theorem and comparison
with the Brownian motion.

Let $X_{a}=a\mathbb{Z}\backslash\R$ and let $x_{0}\in X$. The distribution
of the Brownian random walk starting at $x_{0}$ at time $t$ for
$x\in X$ is $p_{t}(x,x_{0})=\left(\delta_{x_{0}}*f_{t}\right)(x)=\sum_{n\in\Z}f_{t}(x-x_{0})$,
with $f_{t}(x)=\frac{1}{\sqrt{2\pi t}}\exp\left(-x^{2}/2t\right)$. 

By normalizing and choosing $\lambda=a^{2}$, we may consider a fixed
space $X=\mathbb{Z}\backslash\mathbb{R}$, a fixed point $x_{0}=\mathbb{Z}0$
and let $f_{t}^{\lambda}(x)=\frac{1}{\sqrt{2\pi\lambda^{-1}t}}\exp\left(-x^{2}\lambda/2t\right)$.
Then $p_{t}^{\lambda}(x)=\sum_{n\in\mathbb{Z}}f_{t}^{\epsilon}(x+n)$.
\begin{prop}
We have for every $\lambda>0$, $t\ge0$.
\[
\exp\left(-\lambda^{-1}t\right)\le\n{p_{t}^{\lambda}-\pi}\le\sqrt{\frac{2}{1-\exp\left(-2\lambda^{-1}t\right)}}\cdot\exp\left(-\lambda^{-1}t\right).
\]
\end{prop}

The proposition says that the time until $\n{p_{t}^{\lambda}-\pi}\le e^{-T}$
takes place is $\Theta\left(T\cdot\lambda^{-1}\right)$. Therefore
this random walk does not exhibit a cutoff.
\begin{proof}
Let us calculate the Fourier series of $p_{t}^{\lambda}$:
\begin{align*}
\hat{p}_{t}^{\lambda}(m) & =\intop_{0}^{1}p_{t}^{\lambda}(x)\exp(2\pi imx)dx=\\
 & =\intop_{0}^{1}\sum_{n\in\mathbb{Z}}f_{t}^{\lambda}(x+n)\exp(2\pi imx)dx\\
 & =\intop_{-\infty}^{\infty}f_{t}^{\lambda}\exp\left(2\pi imx\right)dx\\
 & =\hat{f}_{t}^{\lambda}(m),
\end{align*}
where $\hat{f}_{t}^{\lambda}$ is the Fourier transform of $f_{t}^{\lambda}$.
By a standard computation $\hat{f}_{t}^{\lambda}(\omega)=\exp\left(-\lambda^{-1}t\omega^{2}\right)$,
so $\hat{p}_{t}^{\lambda}(m)=\exp\left(-\lambda^{-1}tm^{2}\right)$.

On the one hand, 
\begin{align*}
\n{p_{t}^{\lambda}-\pi} & \ge\intop_{0}^{1}\left(p_{t}^{\lambda}(x)-1\right)exp\left(2\pi x\right)dx=\intop_{0}^{1}p_{t}^{\lambda}(x)exp\left(2\pi x\right)dx\\
 & =\hat{p}_{t}^{\lambda}(1)=\exp\left(-\lambda^{-1}t\right).
\end{align*}
On the other hand, 
\begin{align*}
\nn{p_{t}^{\lambda}-\pi}^{2} & =\sum_{m\in\Z}\left(\hat{p}_{t}^{\lambda}(m)-\hat{\pi}(m)\right)^{2}=\sum_{m\in\mathbb{N}\backslash\{0\}}\left(\hat{p}_{t}^{\epsilon}\right)^{2}(m)\\
 & =\sum_{m\in\Z\backslash\{0\}}\exp\left(-2\lambda^{-1}tm^{2}\right)\\
 & =2\sum_{m=1}^{\infty}\exp\left(-2\lambda^{-1}tm\right)\\
 & \le\frac{2}{1-\exp\left(-2\lambda^{-1}t\right)}\exp\left(-2\lambda^{-1}t\right).
\end{align*}
Cauchy-Schwartz inequality completes the proof by
\[
\n{p_{t}^{\epsilon}-\pi}\le\nn{p_{t}^{\epsilon}-\pi}\le\sqrt{\frac{2}{1-\exp\left(-2\lambda^{-1}t\right)}}\exp\left(-\lambda^{-1}t\right).
\]
\end{proof}
\bibliographystyle{amsplain}
\bibliography{../../database}

\providecommand{\bysame}{\leavevmode\hbox to3em{\hrulefill}\thinspace}
\providecommand{\MR}{\relax\ifhmode\unskip\space\fi MR }
\providecommand{\MRhref}[2]{%
  \href{http://www.ams.org/mathscinet-getitem?mr=#1}{#2}
}
\providecommand{\href}[2]{#2}
\begin{thebibliography}{10}

\bibitem{brooks2004random}
Robert Brooks and Eran Makover, \emph{{R}andom construction of {R}iemann
  surfaces}, Journal of Differential Geometry \textbf{68} (2004), no.~1,
  121--157.

\bibitem{cammarota2014asymptotic}
Valentina Cammarota, Alessandro De~Gregorio, and Claudio Macci, \emph{{O}n the
  asymptotic behavior of the hyperbolic {B}rownian motion}, Journal of
  Statistical Physics \textbf{154} (2014), no.~6, 1550--1568.

\bibitem{davidoff2003elementary}
Giuliana Davidoff, Peter Sarnak, and Alain Valette, \emph{{E}lementary number
  theory, group theory and {R}amanujan graphs}, vol.~55, Cambridge University
  Press, 2003.

\bibitem{davies1988heat}
Edward~B Davies and Nikolaos Mandouvalos, \emph{{H}eat kernel bounds on
  hyperbolic space and {K}leinian groups}, Proceedings of the London
  Mathematical Society \textbf{3} (1988), no.~1, 182--208.

\bibitem{decorte2017lower}
Evan DeCorte and Konstantin Golubev, \emph{{L}ower bounds for the measurable
  chromatic number of the hyperbolic plane}, arXiv preprint arXiv:1708.01081
  (2017).

\bibitem{diaconis1988group}
Persi Diaconis, \emph{{G}roup representations in probability and statistics},
  Lecture Notes-Monograph Series \textbf{11} (1988), i--192.

\bibitem{diaconis1996cutoff}
\bysame, \emph{{T}he cutoff phenomenon in finite {M}arkov chains}, Proceedings
  of the National Academy of Sciences \textbf{93} (1996), no.~4, 1659--1664.

\bibitem{friedman2003proof}
Joel Friedman, \emph{{A} proof of {A}lon's second eigenvalue conjecture},
  Proceedings of the thirty-fifth annual ACM symposium on Theory of computing,
  ACM, 2003, pp.~720--724.

\bibitem{gromov1983topological}
Mikhail Gromov and Vitali~D Milman, \emph{{A} topological application of the
  isoperimetric inequality}, American Journal of Mathematics \textbf{105}
  (1983), no.~4, 843--854.

\bibitem{harish1958spherical}
Harish-Chandra, \emph{{S}pherical functions on a semisimple {L}ie group, {I}},
  American Journal of Mathematics (1958), 241--310.

\bibitem{Helgason2000}
Sigurdur Helgason, \emph{{G}roups and geometric analysis, volume 83 of
  {M}athematical {S}urveys and {M}onographs}, American Mathematical Society,
  Providence, RI (2000).

\bibitem{humphries2016density}
Peter Humphries, \emph{{D}ensity {T}heorems for {E}xceptional {E}igenvalues for
  {C}ongruence {S}ubgroups}, arXiv preprint arXiv:1609.06740 (2016).

\bibitem{huntley1993density}
Jonathan Huntley and Yonatan Katznelson, \emph{{D}ensity theorems for
  congruence groups in real rank 1}, Duke Math. J \textbf{71} (1993), 463--473.

\bibitem{huxley1986exceptional}
MN~Huxley, \emph{{E}xceptional eigenvalues and congruence subgroups}, The
  Selberg Trace Formula and Related Topics, Contemp. Math \textbf{53} (1986),
  341--349.

\bibitem{iwaniec1985density}
Henryk Iwaniec, \emph{{D}ensity theorems for exceptional eigenvalues of the
  {L}aplacian for congruence groups}, Banach Center Publications \textbf{17}
  (1985), no.~1, 317--331.

\bibitem{kamber2016lpcomplex}
Amitay Kamber, \emph{{L}p-expander complexes}, arXiv preprint arXiv:1701.00154
  (2016).

\bibitem{kamber2016lpgraph}
\bysame, \emph{{L}p-expander graphs}, arXiv preprint arXiv:1609.04433 (2016).

\bibitem{lax1982asymptotic}
Peter~D Lax and Ralph~S Phillips, \emph{{T}he asymptotic distribution of
  lattice points in {E}uclidean and non-{E}uclidean spaces}, Journal of
  Functional Analysis \textbf{46} (1982), no.~3, 280--350.

\bibitem{lubetzky2017random}
Eyal Lubetzky, Alex Lubotzky, and Ori Parzanchevski, \emph{{R}andom walks on
  {R}amanujan complexes and digraphs}, arXiv preprint arXiv:1702.05452 (2017).

\bibitem{lubetzky2015cutoff}
Eyal Lubetzky and Yuval Peres, \emph{{C}utoff on all {R}amanujan graphs}, arXiv
  preprint arXiv:1507.04725 (2015).

\bibitem{lubotzky1988ramanujan}
Alexander Lubotzky, Ralph Phillips, and Peter Sarnak, \emph{{R}amanujan
  graphs}, Combinatorica \textbf{8} (1988), no.~3, 261--277.

\bibitem{lubotzky2005ramanujan}
Alexander Lubotzky, Beth Samuels, and Uzi Vishne, \emph{{R}amanujan complexes
  of type {A}d}, Israel Journal of Mathematics \textbf{149} (2005), no.~1,
  267--299.

\bibitem{marcus2013interlacing}
Adam Marcus, Daniel~A Spielman, and Nikhil Srivastava, \emph{{I}nterlacing
  families {I}: bipartite {R}amanujan graphs of all degrees}, Foundations of
  Computer Science (FOCS), 2013 IEEE 54th Annual Symposium on, IEEE, 2013,
  pp.~529--537.

\bibitem{parzanchevski2017super}
Ori Parzanchevski and Peter Sarnak, \emph{{S}uper-{G}olden-{G}ates for {PU}
  (2)}, arXiv preprint arXiv:1704.02106 (2017).

\bibitem{sarnak1990diophantine}
Peter Sarnak, \emph{{D}iophantine problems and linear groups}, Proceedings of
  the International Congress of Mathematicians, vol.~1, 1990, pp.~459--471.

\bibitem{sarnak1995selberg}
\bysame, \emph{{S}elberg's eigenvalue conjecture}, Notices of the AMS
  \textbf{42} (1995), no.~11, 1272--1277.

\bibitem{sarnak1991bounds}
Peter Sarnak and Xiao~Xi Xue, \emph{{B}ounds for multiplicities of automorphic
  representations}, Duke Math. J \textbf{64} (1991), no.~1, 207--227.

\bibitem{selberg1965estimation}
Atle Selberg, \emph{{O}n the estimation of {F}ourier coefficients of modular
  forms}, Proc. Sympos. Pure Math, vol.~8, 1965, pp.~1--15.

\bibitem{Terras2013}
Audrey Terras, \emph{{H}armonic analysis on symmetric spaces - {E}uclidean
  space, the sphere, and the {P}oincar{\'e} upper half-plane}, Springer Science
  \& Business Media, 2013.

\end{thebibliography}

\end{document}